\documentclass[a4paper]{article}
\usepackage[utf8]{inputenc}
\usepackage[english]{babel}
\usepackage{amsmath,amsthm,amssymb}
\usepackage{bbm}	%% used for \mathbbm
\usepackage[square,numbers]{natbib}	%% bibliography
\usepackage{graphicx, color, soul} %% soul for highlight
\usepackage[title]{appendix}
\usepackage[paperwidth=19.5cm,paperheight=25cm]{geometry}
\usepackage{url}
\usepackage{blkarray} %matrix with labels
\usepackage[colorlinks=true,linkcolor=blue,citecolor=blue]{hyperref}
\usepackage[capitalize,nameinlink]{cleveref}

\usepackage{authblk} %% authors and affiliation

\usepackage[shortlabels]{enumitem}	%% change style of lists as "[(i)]"
\setlist{labelsep=.25in,leftmargin=*,labelindent=1cm,topsep=2pt,noitemsep}	%% adjust lists
\setlist[enumerate]{label=(\roman*)}

\numberwithin{equation}{section}

\theoremstyle{plain}
\newtheorem{theorem}{Theorem}[section]

\newtheorem{proposition}{Proposition}[section]

\newtheorem*{theorem*}{Theorem}
\theoremstyle{remark}
\newtheorem{remark}{Remark}[section]
\newtheorem*{lemma*}{Lemma}
\newtheorem*{remark*}{Remark}
\theoremstyle{definition}
\newtheorem{example}{Example}[section]

\newcommand{\Polya}{P\'{o}lya }
\newcommand{\limN}{\lim_{n\rightarrow\infty}}

\newcommand{\sumN}{\sum_{i=1}^n}
\newcommand{\sumINF}{\sum_{n=1}^\infty}
\newcommand{\Be}{\textnormal{Beta}} %Beta distribution
\usepackage{geometry}
\title{Sufficientness postulates for measure-valued \Polya urn sequences}
\author[1,2]{ Hristo Sariev\thanks{h.sariev@math.bas.bg; hsariev@uni-sofia.bg}}
\author[2,1]{Mladen Savov\thanks{msavov@fmi.uni-sofia.bg; mladensavov@math.bas.bg}}
\affil[1]{\normalsize Institute of Mathematics and Informatics, Bulgarian Academy of Sciences, 8 Acad. Georgi Bonchev Str., Sofia 1113, Bulgaria}
\affil[2]{\normalsize Faculty of Mathematics and Informatics, Sofia University "St. Kliment Ohridski", 5 James Bourchier Blvd, Sofia 1164, Bulgaria}

\date{}

\begin{document}

\maketitle

\begin{abstract}
In a recent paper, the authors studied the distribution properties of a class of exchangeable processes, called measure-valued \Polya sequences (MVPS), which arise as the observation process in a generalized urn sampling scheme. Here we present several results in the form of ``sufficientness'' postulates that characterize their predictive distributions. In particular, we show that exchangeable MVPSs are the unique exchangeable models whose predictive distributions are a mixture of the marginal distribution and the average of a probability kernel evaluated at past observation. When the latter coincides with the empirical measure, we recover a well-known result for the exchangeable model with a Dirichlet process prior. In addition, we provide a ``pure'' sufficientness postulate for exchangeable MVPSs that does not assume a particular analytic form for the predictive distributions. Two other sufficientness postulates consider the case when the state space is finite.
\end{abstract}


\section{Introduction}

The exchangeable model with a Dirichlet process prior is a cornerstone of Bayesian nonparametric analysis and has inspired much of the subsequent work in the area (for an overview, see \cite{lijoi2010}). One remarkable feature of the Dirichlet process, due to \cite{regazzini1978} and later \cite{lo1991}, is that it is the unique prior for which the predictive distributions of the underlying exchangeable process are a mixture of the marginal distribution and the empirical measure. More specifically, let $(X_n)_{n\geq1}$ be a sequence of random variables taking values in some standard space, say $\mathbb{X}=[0,1]$, and suppose that $(X_n)_{n\geq1}$ is exchangeable, i.e. the order of $X_1,X_2,\ldots$ is probabilistically irrelevant (refer to Section \ref{section:preliminaries} for a rigorous treatment). Denote by $\nu(\cdot)=\mathbb{P}(X_1\in\cdot\,)$ the marginal distribution of the process. If for each $n=1,2,\ldots$ the predictive distribution of $X_{n+1}$ given $(X_1,\ldots,X_n)$ is the convex combination
\begin{equation}\label{intro:equation:dirichlet:prediction}
\mathbb{P}(X_{n+1}\in\cdot\mid X_1,\ldots,X_n)=(1-a_n)\nu(\cdot)+a_n\cdot\frac{1}{n}\sumN\delta_{X_i}(\cdot),
\end{equation}
for some constant $a_n\in(0,1)$, where $\delta_x$ is the unit mass at $x$, then Theorem 1 in \cite{lo1991} implies that $a_n=n/(n+\theta)$, with $\theta=a_1^{-1}-1$, and $(X_n)_{n\geq1}$ has a Dirichlet process prior with parameters $(\theta,\nu)$. The well-known Blackwell-MacQueen urn model \cite{blackwell1973} assumes \eqref{intro:equation:dirichlet:prediction} directly with $a_n=n/(n+\theta)$. 

Predictive constructions of the type \eqref{intro:equation:dirichlet:prediction} have several attractive features. For one, they provide a simple and efficient way to simulate the sequence $X_1,X_2,\ldots$. Indeed, \eqref{intro:equation:dirichlet:prediction} leads to a recursive procedure
\[\mathbb{P}(X_{n+1}\in\cdot\mid X_1,\ldots,X_n)=\frac{\theta+n-1}{\theta+n}\,\mathbb{P}(X_n\in\cdot\mid X_1,\ldots,X_{n-1})+\frac{1}{\theta+n}\,\delta_{X_i}(\cdot),\]
which allows for fast computations and limited storage of information for streaming data (see \cite[Section 5]{fortini2024} for more details on recursive algorithms and predictions). On the other hand, the system of predictive distributions $(\mathbb{P}(X_{n+1}\in\cdot\mid X_1,\ldots,X_n))_{n\geq0}$ determines the law of the process $\mathbb{P}((X_1,X_2,\ldots)\in\cdot\,)$ without the need for prior elicitation. Such a predictive approach to Bayesian modeling, focused on observable quantities instead of unknown parameters, is deeply rooted in the philosophical foundations of Bayesian analysis \cite{fortini2023,fortini2024}, and has recently enjoyed renewed interest, see e.g. \cite{berti2024,fong2023,fortini2024}. 

A particular class of predictive characterizations, to which \eqref{intro:equation:dirichlet:prediction} belongs (see below), are the so-called postulates of ``sufficientness'', dating back to the writings of Cambridge philosopher W. E. Johnson \cite{johnson1932} (for more historical details, see \cite{zabell1982} and \cite[Chapter 5]{scalas2010}). Unlike the usual notions of sufficiency, see e.g. \cite[Section 5]{fortini2000}, sufficientness postulates provide conditions on the functional form of the predictive distributions arising from natural judgments on the inductive process of learning. In other words, model specification is based on general assessments about the form and type of sample information used to make predictions. As an example, the original Johnson's sufficientness postulate assumes an exchangeable multinomial model on $\mathbb{X}=\{1,\ldots,k\}$ with $k$ types such that, for each $n=1,2,\ldots$, the probability of observing type $j\in\mathbb{X}$ on the $(n+1)$th trial
\begin{equation}\label{intro:equation:johnson_sufficientness}
\mathbb{P}(X_{n+1}=j|X_1,\ldots,X_n)=f_{n,j}(N_{n,j})
\end{equation}
depends on the sample $(X_1,\ldots,X_n)$ only by the number of times $N_{n,j}$ that $j$ has been observed in the past. By Theorem 2.1 in \cite{zabell1982}, any exchangeable process with predictive distributions \eqref{intro:equation:johnson_sufficientness} has a Dirichlet distribution prior (see also Proposition \ref{result:k-color} in Section \ref{section:results}). Similar considerations about the relevance of different pieces of sample information have led to sufficientness postulates for the two-parameter Poisson-Dirichlet process \cite{zabell1997}, Gibbs-type priors \cite{bacallado2017} and feature-sampling models \cite{camerlenghi2021}. Sufficientness postulates have also been proposed for neutral to the right priors \cite{walker1999} and recurrent Markov exchangeable sequences \cite{zabell1995}.

Within this paradigm, the predictive rule \eqref{intro:equation:dirichlet:prediction} can be shown to be motivated by the same principle behind \eqref{intro:equation:johnson_sufficientness}, extended to more general state spaces $\mathbb{X}$, without necessarily assuming that the number of observed types is limited or known in advance. To that end, let $K_n$ and $X_1^*,\ldots,X_{K_n}^*$ denote, respectively, the number and the distinct values of $(X_1,\ldots,X_n)$ in \textit{order of appearance}. Following \cite{zabell1997}, we assume that the probability of observing the $j$th value to appear
\begin{equation}\label{intro:equation:species_sampling:eq1}
\mathbb{P}(X_{n+1}=X_j^*|X_1,\ldots,X_n)=p_n(N_{n,j})
\end{equation}
is a function of the number of times it has been previously observed, while the probability of observing (discovering) a new type depends only on the sample size,
\begin{equation}\label{intro:equation:species_sampling:eq2}
\mathbb{P}(X_{n+1}\notin\{X_1^*,\ldots,X_{K_n}^*\}|X_1,\ldots,X_n)=r_n.
\end{equation}
Note that unlike \eqref{intro:equation:johnson_sufficientness}, the function $p_n$ does not change depending on $j$, which is a consequence of the framework, since type $j$ is not known to exist in advance \cite[p.381, Footnote 5]{zabell1997}. Moreover, the original proposal by \cite{zabell1997} allows $r_n=r_n(K_n)$ to depend further on the number of distinct values, leading to a general characterization result for the two-parameter Poisson-Dirichlet process. On the other hand, the model \eqref{intro:equation:species_sampling:eq1}-\eqref{intro:equation:species_sampling:eq2} describes an (exchangeable) random partition process (in fact, the Chinese restaurant process and the related Ewens' partition structure \cite[p.380, Remarks]{zabell1997}), but does not discuss the mechanism of `labeling' $X_1^*,X_2^*,\ldots$. Assuming that their values are assigned independently, from a \textit{diffuse} distribution $\nu$, i.e. $\nu(\{x\})=0$ for all $x\in\mathbb{X}$, we can express \eqref{intro:equation:species_sampling:eq1}-\eqref{intro:equation:species_sampling:eq2} jointly as
\begin{equation}\label{intro:equation:species_sampling}
\mathbb{P}(X_{n+1}\in\cdot\mid X_1,\ldots,X_n)=\sum_{j=1}^{K_n}p_n(N_{n,j})\delta_{X_j^*}(\cdot)+r_n\nu(\cdot).
\end{equation}
Exchangeable processes with predictive distributions of the kind \eqref{intro:equation:species_sampling} belong to the class of \textit{species sampling sequences} \cite{hansen2000,pitman1996} and are closely related to the theory of exchangeable random partitions (for an excellent overview, see \cite[Chapter 14]{vaart2017}). It follows from \eqref{intro:equation:species_sampling} that, for any $k=1,\ldots,n$ and all counts $n_1,\ldots,n_k\in\{1,\ldots,n\}$ such that $\sum_{j=1}^kn_j=n$,
\[\sum_{j=1}^{k}p_n(n_j)+r_n=1;\]
thus, assuming that every partition of $(X_1,\ldots,X_n)$ by type has a positive probability of occurring, we can easily deduce that
\[p_n(n_j)=(1-r_n)\frac{n_j}{n}.\]
As a consequence,
\begingroup\allowdisplaybreaks
\begin{align*}
\mathbb{P}(X_{n+1}\in\cdot\mid X_1,\ldots,X_n)&=(1-r_n)\cdot\frac{1}{n}\sum_{j=1}^{K_n}N_{n,j}\delta_{X_j^*}(\cdot)+r_n\nu(\cdot)=(1-r_n)\cdot\frac{1}{n}\sum_{i=1}^{n}\delta_{X_i}(\cdot)+r_n\nu(\cdot),
\end{align*}
\endgroup
suggesting that the analytic form in \eqref{intro:equation:dirichlet:prediction} arises from a ``pure'' sufficientness principle.

An important limitation of \eqref{intro:equation:dirichlet:prediction} and species sampling models, in general, is that they are designed to work with categorical data, taking into account only the frequencies of the observed types. In this work, we study an extension of \eqref{intro:equation:dirichlet:prediction} by replacing $\delta_x$ with a general probability measure $R_x$, so that
\begin{equation}\label{intro:equation:prediction}
\mathbb{P}(X_{n+1}\in\cdot\mid X_1,\ldots,X_n)=(1-a_n)\nu(\cdot)+a_n\cdot\frac{1}{n}\sumN R_{X_i}(\cdot),
\end{equation}
for some constant $a_n\in(0,1)$. Depending on the particular $R$, the predictive rule \eqref{intro:equation:prediction} can more efficiently exploit sample information coming from continuous data by considering, for example, where each observation falls inside $\mathbb{X}$. In fact, for diffuse $R$ and $\nu$, all $X_1,\ldots,X_n$ will be distinct, so notions like random partition and species discovery become practically meaningless. Our goal is to characterize the class of exchangeable processes that are governed by \eqref{intro:equation:prediction}, and thus motivate their use in Bayesian nonparametric analysis from an inductive learning perspective. 

Notice that when $a_n=n/(n+\theta)$ for some $\theta>0$, equation \eqref{intro:equation:prediction} becomes
\[\mathbb{P}(X_{n+1}\in\cdot\mid X_1,\ldots,X_n)=\frac{\theta\nu(\cdot)+\sumN R_{X_i}(\cdot)}{\theta+n},\]
and, following \cite{sariev2023b}, the process $(X_n)_{n\geq1}$ is a member of the class of \textit{measure-valued \Polya sequences} (MVPS). We might expect from \eqref{intro:equation:dirichlet:prediction}, where $R_x=\delta_x$, that the only exchangeable processes having predictive rules \eqref{intro:equation:prediction} will be exchangeable MVPSs, in which case \cite{sariev2023b} provides a complete description of their prior distributions. In particular, it would follow from Theorem 3.13 in \cite{sariev2023b} that when $\mathbb{X}$ is countable or $R$ is dominated by $\nu$, the exchangeable sequence $(X_n)_{n\geq1}$ has a Dirichlet process mixture prior (see Section \ref{section:preliminaries:mvps}). Our main result, Theorem \ref{result:main}, states that this is indeed the case, but the proof is fundamentally different from that for \eqref{intro:equation:dirichlet:prediction}, making extensive use of the measure-valued nature of \eqref{intro:equation:prediction} as the theory of random partitions is generally not available.

The rest of the paper is structured as follows. In the next section, we introduce MVPSs and provide notation that will be used later on. Theorem \ref{result:main} is stated in Section \ref{section:results} along with a ``pure'' sufficientness postulate for exchangeable MVPSs that does not assume a particular analytic form for the predictive distributions, as does \eqref{intro:equation:prediction}. In addition, we provide two other sufficientness postulates that generalize \eqref{intro:equation:johnson_sufficientness} for the case when $\mathbb{X}$ is finite. Proofs are postponed to Section \ref{section:proof}.

\section{The model}\label{section:preliminaries}

\subsection{Preliminaries}

From now on, $\mathbb{X}$ is a complete separable metric space endowed with the Borel $\sigma$-algebra $\mathcal{X}$, which from standard results is countably generated. All random quantities are defined on a common probability space $(\Omega,\mathcal{H},\mathbb{P})$ unless otherwise specified. For all unexplained measure-theoretic details, we refer to \cite{kallenberg2021}.

Given two arbitrary measures $\mu,\pi$ on $\mathbb{X}$, we say that $\mu$ is \textit{absolutely continuous} with respect to $\pi$, denoted $\mu\ll\pi$, if $\mu(A)=0$ whenever $\pi(A)=0$ for any $A\in\mathcal{X}$, and we say that $\mu$ and $\pi$ are \textit{mutually singular}, denoted $\mu\perp\pi$, if there exists $S\in\mathcal{X}$ such that $\mu(S)=0=\pi(S^c)$.

A \textit{transition kernel on $\mathbb{X}$} is a function $R:\mathbb{X}\times\mathcal{X}\rightarrow\mathbb{R}_+$ such that $(i)$ the map $x\mapsto R(x,B)\equiv R_x(B)$ is $\mathcal{X}$-measurable, for all $B\in\mathcal{X}$; and $(ii)$ $R_x$ is a measure on $\mathbb{X}$. In addition, $R$ is said to be \textit{finite} if $R_x(\mathbb{X})<\infty$, and is called a \textit{probability kernel} if $R_x(\mathbb{X})=1$, for all $x\in\mathbb{X}$.

Let $\nu$ be a probability measure on $\mathbb{X}$, and $\mathcal{G}\subseteq\mathcal{X}$ a sub-$\sigma$-algebra on $\mathbb{X}$. A probability kernel $R$ on $\mathbb{X}$ is said to be a \textit{regular version of the conditional distribution (r.c.d.) for $\nu$ given $\mathcal{G}$}, which for emphasis we will denote by
\[R_x(\cdot)=\nu(\cdot\mid\mathcal{G})(x)\qquad\mbox{for }\nu\mbox{-almost every (a.e.) }x,\]
if $(i)$ the map $x\mapsto R_x(B)$ is $\mathcal{G}$-measurable, for all $B\in\mathcal{X}$; and $(ii)$ $\int_AR_x(B)\nu(dx)=\nu(A\cap B)$, for all $A\in\mathcal{G}$ and $B\in\mathcal{X}$. The assumptions on $(\mathbb{X},\mathcal{X})$ guarantee that a r.c.d. for $\nu$ given $\mathcal{G}$ exists and is unique up to a $\nu$-null set, and all conditional distributions must be understood as regular versions. Moreover, a regular version $R$ of $\nu(\cdot\mid\mathcal{G})$ is called \textit{a.e. proper} if there exists $F\in\mathcal{G}$ such that $\nu(F)=1$ and
\[R_x(A)=\delta_x(A)\qquad\mbox{for all }A\in\mathcal{G}\mbox{ and }x\in F.\]
Following \cite[p.649]{berti2007}, an a.e. proper r.c.d. for $\nu$ given $\mathcal{G}$ exists if and only if $\mathcal{G}$ is \textit{countably generated (c.g.) under $\nu$}, in the sense that there exists $C\in\mathcal{G}$ such that $\nu(C)=1$ and $\mathcal{G}\cap C$ is c.g.

\vspace{2mm}

An $\mathbb{X}$-valued sequence $(X_n)_{n\geq1}$ of random variables is said to be \textit{exchangeable} if, for each $n=1,2,\ldots$ and all permutations $\sigma$ of $\{1,\ldots,n\}$,
\[(X_1,\ldots,X_n)\overset{d}{=}(X_{\sigma(1)},\ldots,X_{\sigma(n)}).\]
In that case, by de Finetti representation theorem for exchangeable sequences, see e.g. \cite[Section 3]{aldous1985}, there exists a random probability measure $\tilde{P}$ on $\mathbb{X}$, called the \textit{directing random measure} of the sequence, such that, given $\tilde{P}$, the random variables $X_1,X_2,\ldots$ are conditionally independently and identically distributed (i.i.d.) according to $\tilde{P}$,
\begin{eqnarray}
X_n\mid\tilde{P}& \overset{i.i.d.}{\sim} & \tilde{P} \nonumber\\
\tilde{P}& \sim & Q \nonumber
\end{eqnarray}
where $Q$ is the (nonparameteric) \textit{prior} distribution of $\tilde{P}$. Moreover, $\tilde{P}$ is the almost sure (a.s.) weak limit of the empirical measure
\[\frac{1}{n}\sumN\delta_{X_i}\overset{w}{\longrightarrow}\tilde{P}\qquad\mbox{a.s.},\]
and the predictive distributions of $(X_n)_{n\geq1}$,
\[\mathbb{P}(X_{n+1}\in\cdot\mid X_1,\ldots,X_n)\overset{w}{\longrightarrow}\tilde{P}\qquad\mbox{a.s.}\]

The most popular prior distribution in Bayesian nonparametric analysis is arguably the \textit{Dirichlet process}, $\mbox{DP}(\theta,\nu)$, which is parameterized by a probability measure $\nu$ on $\mathbb{X}$ and a constant $\theta>0$. We recall that $\tilde{P}\sim\mbox{DP}(\theta,\nu)$ if, for every finite partition $\{B_1,\ldots,B_n\}\subseteq\mathcal{X}$ of $\mathbb{X}$, the random vector $(\tilde{P}(B_1),\ldots,\tilde{P}(B_n))$ has a Dirichlet distribution with parameters $(\theta\nu(B_1),\ldots,\theta\nu(B_n))$. Equivalently, see e.g. \cite[Theorem 4.12]{vaart2017}, the Dirichlet process is the distribution of an a.s. discrete random probability measure with so-called stick-breaking weights,
\begin{equation}\label{model:prelim:stick_breaking}
\sum_{j=1}^\infty V_j\delta_{Z_j}\sim\mbox{DP}(\theta\nu),
\end{equation}
where $V_j=W_j\prod_{i=1}^{j-1}W_i$, and $W_1,W_2,\ldots\overset{i.i.d.}{\sim}\Be(1,\theta)$ and $Z_1,Z_2,\ldots\overset{i.i.d.}{\sim}\nu$ are independent random variables.

Exchangeable processes having directing random measures $\tilde{P}$ as in \eqref{model:prelim:stick_breaking}, for arbitrary sequences of non-negative random weights $(V_j)_{j\geq1}$ and independent atoms $(Z_j)_{j\geq1}$ coming from a diffuse distribution $\nu$, are called (proper) \textit{species sampling sequences} and have predictive distributions of the form \eqref{intro:equation:species_sampling} for more general functions $p_n$ and $r_n$, see \cite{hansen2000,pitman1996} for a complete characterization. Because of the discrete nature of their $\tilde{P}$ and the fact that the atoms $Z_1,Z_2,\ldots$ are a.s. all distinct, species sampling models are, by Kingman's theory, in one-to-one correspondence with exchangeable random partitions, up to a specification of $\nu$ (for more information, we refer to \cite{pitman1996} and \cite[Chapter 14]{vaart2017}). In fact, the exchangeable partition generated by a random sample from a Dirichlet process \eqref{model:prelim:stick_breaking} with diffuse $\nu$ is the well-known Chinese restaurant process.

\subsection{Exchangeable MVPS}\label{section:preliminaries:mvps}

Following \cite{sariev2023b}, we call any sequence $(X_n)_{n\geq1}$ of $\mathbb{X}$-valued random variables a \textit{measure-valued \Polya (urn) sequence} with parameters $\theta,\nu$ and $R$, denoted MVPS$(\theta,\nu,R)$, if $\nu(\cdot)=\mathbb{P}(X_1\in\cdot\,)$ and, for each $n=1,2,\ldots$,
\begin{equation}\label{model:mvps:predictive}
\mathbb{P}(X_{n+1}\in\cdot\mid X_1,\ldots,X_n)=\frac{\theta\nu(\cdot)+\sumN R_{X_i}(\cdot)}{\theta+\sumN R_{X_i}(\mathbb{X})},
\end{equation}
where $R$ is a finite transition kernel on $\mathbb{X}$, called the \textit{reinforcement kernel} of the process, $\nu$ is a probability measure on $\mathbb{X}$, known as the \textit{base measure}, and $\theta>0$ is a positive constant. When $R_x(\mathbb{X})=m$ for $\nu$-a.e. $x$ and some $m>0$, both $R$ and the corresponding MVPS are said to be \textit{balanced}. 

The predictive structure \eqref{model:mvps:predictive} admits a physical interpretation in terms of a generalized urn sampling scheme as follows. Suppose initially that we have an urn whose contents are described by the measure $\theta\nu$ in the sense that, for each $B\in\mathcal{X}$, the quantity $\theta\nu(B)$ records the total \textit{mass} of balls whose colors lie in $B$. At time $n=1$, a ball is drawn at random from the urn with probability distribution $\nu$. Depending on its color, $X_1$, we reinforce the urn with additional balls according to the measure $R_{X_1}$, so that the updated composition becomes $\theta\nu+R_{X_1}$. At each subsequent time $n>1$, a ball with color $X_n$ is drawn with probability proportional to $\theta\nu+\sum_{i=1}^{n-1}R_{X_i}$, and the contents of the urn get ``reinforced'' by $R_{X_n}$. The model with $|\mathbb{X}|=k$ and $R_x=\delta_x$, i.e. we only add one additional ball of the observed color, corresponds to the classical $k$-color \Polya urn scheme and measure-valued \Polya urn processes have been proposed by \cite{thacker2022,fortini2021,janson2019,mailler2017} as extensions to infinite colors and more general reinforcement mechanisms. The case where the observation process $(X_n)_{n\geq1}$ is exchangeable has been investigated by \cite{sariev2023b}, and the main results from that paper are summarized in the following theorem.

\begin{theorem}[Theorems 3.2, 3.7 and 3.9 in \cite{sariev2023b}]\label{model:theorem}
Let $(X_n)_{n\geq1}$ be an exchangeable, but not i.i.d. MVPS with parameters $(\theta,\nu,R)$ and directing random measure $\tilde{P}$. Then
\begin{enumerate}
\item[(a)] $R$ is balanced and there exists a (c.g. under $\nu$) sub-$\sigma$-algebra $\mathcal{G}$ of $\mathcal{X}$ such that, normalized, $R$ is a r.c.d. for $\nu$ given $\mathcal{G}$,
\[\frac{R_x(\cdot)}{R_x(\mathbb{X})}=\nu(\cdot\mid\mathcal{G})(x)\qquad\mbox{for }\nu\mbox{-a.e. }x;\]
\item[(b)] $\mathbb{P}(X_{n+1}\in\cdot\mid X_1,\ldots,X_n)$ converges a.s. in total variation, as $n\rightarrow\infty$, to $\tilde{P}$;
\item[(c)] $\tilde{P}$ is equal in law to
\begin{equation}\label{model:mvps:prior}
\sum_{j=1}^\infty V_j\frac{R_{Z_j}(\cdot)}{R_{Z_j}(\mathbb{X})},
\end{equation}
where $(V_j)_{j\geq1}$ and $(Z_j)_{j\geq1}$ are given in \eqref{model:prelim:stick_breaking}.
\end{enumerate}
\end{theorem}

Independently from the theory on measure-valued urn processes, Berti et al. \cite{berti2023} introduced the \textit{kernel based Dirichlet sequence}, which is a balanced MVPS with reinforcement kernel $R(\cdot)=\nu(\cdot\mid\mathcal{G})$, for some sub-$\sigma$-algebra $\mathcal{G}$ of $\mathcal{X}$. In \cite{berti2023}, the authors prove that kernel based Dirichlet sequences are exchangeable, having directing random measures \eqref{model:mvps:prior}. Theorem \ref{model:theorem} provides a converse statement and shows that, under exchangeability, $R$ need not be assumed to be balanced a priori. In fact, it follows from Theorem \ref{model:theorem} that any exchangeable MVPS$(\theta,\nu,R)$ can be reparameterized as a balanced exchangeable MVPS$(\tilde{\theta},\nu,\tilde{R})$ for some probability kernel $\tilde{R}$ and constant $\tilde{\theta}$, so that
\begin{equation}
\label{model:equation:mvps:sufficientness}
\mathbb{P}(X_{n+1}\in\cdot\mid X_1,\ldots,X_n)=\frac{\tilde{\theta}}{\tilde{\theta}+n}\nu(\cdot)+\frac{n}{\tilde{\theta}+n}\cdot\frac{1}{n}\sumN\tilde{R}_{X_i}(\cdot).
\end{equation}
Therefore, exchangeable MVPSs satisfy the sufficientness postulate \eqref{intro:equation:prediction}, and in Section \ref{section:results} we show that the opposite statement is also true. On the other hand, by Theorem \ref{model:theorem}\textit{(c)}, the directing random measure $\tilde{P}$ of any MVPS inherits its properties directly from the reinforcement kernel $\tilde{R}$, so that $\tilde{P}$ is a.s. diffuse and/or absolutely continuous with respect to some reference probability measure whenever $\tilde{R}$ is. This is in stark contrast to the discrete nature of the Dirichlet process \eqref{model:prelim:stick_breaking} and species sampling models, in general, all while keeping the same simple form of the predictive distributions \eqref{model:equation:mvps:sufficientness}.

Theorem \ref{model:theorem} can be refined when $\mathbb{X}$ is countable or $R$ is dominated by $\nu$, in which case exchangeable MVPSs turn out to have Dirichlet process mixture priors.

\begin{theorem}[Theorems 3.10 and 3.13 in \cite{sariev2023b}]\label{model:theorem:dominated}
Let $(X_n)_{n\geq1}$ be as in Theorem \ref{model:theorem}. If $R_x\ll\nu$ for $\nu$-a.e. $x$, then
\begin{enumerate}
\item[(i)] there exists a countable partition $D_1,D_2,\ldots$ of $\mathbb{X}$ in $\mathcal{X}$ such that
\[\frac{R_x(\cdot)}{R_x(\mathbb{X})}=\sum_{k\geq1}\nu(\cdot\mid D_k)\cdot\mathbbm{1}_{D_k}(x)\qquad\mbox{for }\nu\mbox{-a.e. }x;\]
\item[(ii)] letting $m=R_x(\mathbb{X})$, 
\begin{eqnarray}
X_n\mid\tilde{P} & \overset{i.i.d.}{\sim} & \sum_{k\geq1}\tilde{P}(D_k)\,\nu(\cdot\mid D_k) \nonumber\\
\bigl(\tilde{P}(D_k)\bigr)_{k\geq1} & \sim & \textnormal{DP}\Bigl(\frac{\theta}{m},\nu\Bigr) \nonumber
\end{eqnarray}
\end{enumerate}
\end{theorem}

\section{Results}\label{section:results}

Theorem \ref{result:main}, together with the discussion around \eqref{model:equation:mvps:sufficientness}, states that the family of exchangeable MVPSs coincides with the class of exchangeable processes satisfying the sufficientness postulate \eqref{intro:equation:prediction}.

\begin{theorem}\label{result:main}
Let $(X_n)_{n\geq1}$ be an exchangeable process with marginal distribution $\nu$ and predictive distributions \eqref{intro:equation:prediction}, that is, for each $n=1,2,\ldots$,
\begin{equation}\label{equation:predictive_distributions}
\mathbb{P}(X_{n+1}\in\cdot\mid X_1,\ldots,X_n)=(1-a_n)\nu(\cdot)+a_n\cdot\frac{1}{n}\sumN R_{X_i}(\cdot),
\end{equation}
where $R$ is a probability kernel on $\mathbb{X}$, and $a_n\in(0,1)$. Then $a_n=n/(a_1^{-1}-1+n)$ for all $n\geq1$. Moreover, $(X_n)_{n\geq1}$ is an MVPS with parameters $(a_1^{-1}-1,\nu,R)$.
\end{theorem}

\begin{remark}\label{remark:characterization}
Since $(X_n)_{n\geq1}$ in Theorem \ref{result:main} is an exchangeable MVPS, then Theorems \ref{model:theorem} and \ref{model:theorem:dominated} provide additional guidance for the choice of the parameter $R$. In particular, there exists a (c.g. under $\nu)$ sub-$\sigma$-algebra $\mathcal{G}$ of $\mathcal{X}$ such that $R$ is a r.c.d. for $\nu$ given $\mathcal{G}$,
\[R_x(\cdot)=\nu(\cdot\mid\mathcal{G})(x)\qquad\mbox{for }\nu\mbox{-a.e. }x.\]
Therefore, in a Bayesian nonparametric setting, when we make the initial assessment that our exchangeable model is characterized by \eqref{equation:predictive_distributions}, instead of choosing a specific probability kernel $R$, we actually have to decide on a predefined body of knowledge in the form of the $\sigma$-algebra $\mathcal{G}$, making the model ``essentially'' parameterized by $\mathcal{G}$. If, in addition, $R$ is dominated by $\nu$, the above amounts to a decision on a particular partition of $\mathbb{X}$ (see Examples \ref{example:mixture} and \ref{example:k_color}). 

%On the other hand, if in Theorem \ref{result:main} we assume from the outset that the reinforcement measure $R$ has the form \eqref{equation:reinforcement:representation}, then it can be shown, arguing as in Theorem 8 in \citep{berti2023}, that $(R_{X_n})_{n\geq1}$ is an exchangeable process with predictive distributions
%\[\mathbb{P}(R_{X_{n+1}}\in\cdot\mid R_{X_1},\ldots,R_{X_n})=(1-a_n)(\nu\circ R^{-1})(\cdot)+a_n\cdot\frac{1}{n}\sumN\delta_{R_{X_i}}(\cdot)\qquad\mbox{a.s.},\]
%where $\nu\circ R^{-1}$ is the image measure of $\nu$ under $x\mapsto R_x$. In that case, Theorem 1 in \cite{lo1991} implies that $a_n=n/(n+\theta)$ with $\theta:=a_1^{-1}-1$, for $n=1,2,\ldots$, and hence
%\[\mathbb{P}(X_{n+1}\in\cdot\mid X_1,\ldots,X_n)=\frac{\theta\nu(\cdot)+\sumN R_{X_i}(\cdot)}{\theta+n}\qquad\mbox{a.s.}\]
\end{remark}

\begin{example}[Mixture model]\label{example:mixture}
Let $\theta>0$ and $\nu,P_1,P_2,\ldots$ be probability measures on $\mathbb{X}$ such that $P_k(D_k)=1$, $k\geq1$, where $D_1,D_2,\ldots\in\mathcal{X}$ forms a countable partition of $\mathbb{X}$. For each $x\in\mathbb{X}$, we define
\[P_{(x)}(\cdot):=P_k(\cdot)\qquad\mbox{if and only if}\qquad x\in D_k.\]
Let $(X_n)_{n\geq1}$ be an exchangeable MVPS with predictive distributions
\[\mathbb{P}(X_{n+1}\in\cdot\mid X_1,\ldots,X_n)=\frac{\theta}{\theta+n}\nu(\cdot)+\frac{n}{\theta+n}\cdot\frac{1}{n}\sumN P_{(X_i)}(\cdot).\]
By Theorem 3.14 in \cite{sariev2023b}, the reinforcement kernel of any MVPS admits a Lebesgue decomposition with respect to a universal set, so in the present example, there exist sets $C,S\in\mathcal{X}$ such that $\nu(C)=1$,
\[P_{(x)}\perp\nu\quad\mbox{and}\quad P_{(x)}(S)=1,\qquad\mbox{for every }x\in C\cap S,\]
and
\[P_{(x)}\ll\nu\quad\mbox{and}\quad P_{(x)}(S^c)=1,\qquad\mbox{for every }x\in C\cap S^c.\]
Let $\{D_{j_1},D_{j_2},\ldots\}$ be the smallest cover of $C\cap S^c$ extracted from $\{D_1,D_2,\ldots\}$. Since $P_{(x)}\equiv P_k$ when $x\in D_k$, then $P_{(x)}\ll\nu$ and $P_{(x)}(\cup_{i\geq1}D_{j_i})=1$, for all $x\in\bigcup_{i\geq1}D_{j_i}$. Furthermore, $\nu(D_{j_i})>0$, for each $i=1,2,\ldots$, and
\[\mathbb{P}(X_{n+1}\in\cup_{i\geq1}D_{j_i}|X_1,\ldots,X_n\bigr)\geq\frac{\theta}{\theta+n}\nu(\cup_{i\geq1}D_{j_i}),\]
which implies that $\sumINF\mathbb{P}(X_{n+1}\in\cup_{i\geq1}D_{j_i}|X_1,\ldots,X_n)=\infty$ and, from the conditional Borel-Cantelli lemma, $\mathbb{P}(X_n\in\cup_{i\geq1}D_{j_i}\mbox{ i.o.})=1$. It follows that the subsequence of $(X_n)_{n\geq1}$ restricted on $\bigcup_{i\geq1}D_{j_i}$ is itself an exchangeable MVPS with parameters $\theta^*=\theta\nu(\cup_{i\geq1}D_{j_i})$, $\nu^*(\cdot)=\nu(\cdot\,|\cup_{i\geq1}D_{j_i})$ and $R_x=P_{(x)}$, for $x\in\bigcup_{i\geq1}D_{j_i}$, so by Theorem \ref{model:theorem:dominated}, as applied to the subsequence, we get
%% Argue as in Sariev et al. (2023) that the predictive distributions of the subsequence have the form of a balanced MVPS. Then apply Theorem 7 in Berti et al. (2023) to show that it is exchangeable. By Theorem 3.10 in \citep{sariev2023b}, there exists a partition $\{C_1,C_2,\ldots\}$, but it must coincide with $\{D_{j_1},D_{j_2},\ldots\}$
\[P_{j_i}(\cdot)=P_{(x)}(\cdot)=\nu^*(\cdot\mid D_{j_i})=\nu(\cdot\mid D_{j_i}),\]
using some $x\in D_{j_i}$ as reference. The singular part of the reinforcement kernel has no such simple representation, yet by Theorem \ref{model:theorem} it is a r.c.d. for $\nu$ on $S$.

If $\nu(S)=0$ or, equivalently, $P_k\ll\nu$ for all $k\geq1$, then
\[P_{(x)}(\cdot)=\sum_{k\geq1}\nu(\cdot\mid D_k)\cdot\mathbbm{1}_{D_k}(x)\qquad\mbox{for every }x\in\mathbb{X},\]
and, from Theorem \ref{model:theorem:dominated},
\begin{equation}
\begin{aligned}
X_n\mid\tilde{P}&\quad\overset{i.i.d.}{\sim}&&\sum_{k\geq1}\tilde{P}(D_k)P_k(\cdot) \nonumber\\
\bigl(\tilde{P}(D_k)\bigr)_{k\geq1}&\quad\;\,\sim&&\mbox{DP}(\theta,\nu) \nonumber
\end{aligned}
\end{equation}
Such models are commonly used for density estimation, see e.g. \cite[Chapter 5]{vaart2017}, and exchangeable MVPSs with dominated reinforcement kernels offer a particular set of model choices. For example, when $|\mathbb{X}|<\infty$ and $\nu(A)=\frac{|A|}{|\mathbb{X}|}$, $A\in\mathcal{X}$, is uniform on $\mathbb{X}$, the directing random measure of $(X_n)_{n\geq1}$ becomes
\[\sum_{k\geq1}\tilde{P}(D_k)\frac{|\,\cdot\,\cap\,D_k|}{|D_k|},\]
which is a random histogram with non-uniform widths given by the $|D_k|$'s. At the other extreme, when $\nu(dx)=f(x)dx$ is absolutely continuous (with respect to the Lebesgue measure) on $\mathbb{X}=\mathbb{R}$, the directing random measure of $(X_n)_{n\geq1}$ is itself absolutely continuous with random density
\[\sum_{k\geq1}\tilde{P}(D_k)f_k(x),\]
where $f_k(x)=\frac{\mathbbm{1}_{D_k}(x)f(x)}{\int_{D_k}f(y)dy}$, $x\in\mathbb{X}$, is the section of $f$ over $D_k$, normalized.
\end{example}

\begin{example}($k$-color urn model)\label{example:k_color}
Let $(X_n)_{n\geq1}$ be an exchangeable MVPS$(\theta,\nu,R)$ on $\mathbb{X}=\{1,\ldots,k\}$, where, without loss of generality, we assume that $R$ is a probability kernel (see Section \ref{section:preliminaries:mvps}). By Theorem  \ref{model:theorem:dominated}, there exists a partition $\{D_1,\ldots,D_m\}\subseteq\mathcal{X}$ of $\mathbb{X}$, for some $1\leq m\leq k$, such that
\begin{equation}\label{example:k_color:representation}
R_x(\cdot)=\sum_{l=1}^m\nu(\cdot\mid D_l)\cdot\mathbbm{1}_{D_l}(x)\qquad\mbox{for every }x\in\mathbb{X}.
\end{equation}
It follows that, with $\nu$ and $\theta$ fixed, the total number of exchangeable $k$-color MVPSs is equal to the number of possible partitions of $\mathbb{X}$, also known as Bell's number, which is given by the formula
\[B_k=\sum_{j=0}^k\frac{1}{j!}\sum_{i=0}^j(-1)^{j-i}\binom{j}{i}i^k,\]
with $B_2=2$, $B_3=5$, $B_4=15$, etc. The classical $k$-color \Polya urn model corresponds to the case with the finest partition $D_j=\{j\}$, $j=1,\ldots,k$, while we obtain an i.i.d. sequence when $D_1=\mathbb{X}$. On the other hand, intermediate cases can be appealing from an application point of view and can be used, for example, in a clinical trial setting to target both fixed asymptotic allocations (within the $D_j$'s) and random asymptotic allocations (across the $D_j$'s), simultaneously.
\end{example}

\subsection{``Pure'' sufficientness postulate for MVPSs}\label{section:results:pure}

Recall from the introduction that \eqref{intro:equation:dirichlet:prediction} is equivalent to \eqref{intro:equation:species_sampling:eq1}-\eqref{intro:equation:species_sampling:eq2}, up to a relabeling of the observations. In this section, we investigate whether exchangeable MVPSs can be characterized by statements like those in \eqref{intro:equation:species_sampling:eq1}-\eqref{intro:equation:species_sampling:eq2}, which one might argue reflect a ``pure'' sufficientness postulate since they do not assume a particular analytic form for the predictive distributions. Based on the discussion in Remark \ref{remark:characterization}, such a result should be stated in terms of the sub-$\sigma$-algebra $\mathcal{G}$ for which $R(\cdot)=\nu(\cdot\mid\mathcal{G})$. Here, we follow the approach by \cite[Proposition 3.3]{fortini2024} (see also Example \ref{pure:example:dirichlet_process}), who give an alternative formulation of \eqref{intro:equation:species_sampling:eq1}-\eqref{intro:equation:species_sampling:eq2} via predictive sufficiency. To that end, we introduce, for each $n=1,2,\ldots$ and every $A\in\mathcal{X}$, the notation
\[N_{n,A}:=\sumN\mathbbm{1}_A(X_i).\]
We also define, for each $n=1,2,\ldots$,
\[\mathcal{G}_n:=X_n^{-1}(\mathcal{G}).\]

The next theorem states that an exchangeable process $(X_n)_{n\geq1}$ is an MVPS with reference to some c.g. under $\nu$  sub-$\sigma$-algebra $\mathcal{G}$ of $\mathcal{X}$ if $(i)$ the probability that $X_{n+1}$ belongs to a set $A\in\mathcal{G}$ depends only on the number of observations $N_{n,A}$ in $A$; and $(ii)$, given that it belongs to $\mathcal{G}$, the random variable $X_{n+1}$ is conditionally independent from $(X_1,\ldots,X_n)$. Note that condition $(i)$ becomes vacuous when $\mathcal{G}$ is a simple $\sigma$-algebra generated by a single subset $A\in\mathcal{X}$ (see also Remark 1 in \cite{zabell1982} and the discussion after Proposition \ref{result:k-color}). For that purpose, we introduce the notion of a \textit{$\nu$-essentially not simple} $\sigma$-algebra, which is a $\sigma$-algebra $\mathcal{G}$ containing a finite partition $\{B_1,\ldots,B_k\}\subseteq\mathcal{G}$ of $\mathbb{X}$ with $|k|>2$ elements such that at least three of them have positive $\nu$-probability, $\#\{j:\nu(B_j)>0\}>2$.

\begin{theorem}\label{pure:result:main}
Let $(X_n)_{n\geq1}$ be an exchangeable process with marginal distribution $\nu$, and let $\mathcal{G}$ be a sub-$\sigma$-algebra of $\mathcal{X}$ on $\mathbb{X}$, which is c.g. under $\nu$ and $\nu$-essentially not simple. If for each $n=1,2,\ldots$ it holds
\begin{equation}\label{pure:result:main:eq1}
\mathbb{P}(X_{n+1}\in A|X_1,\ldots,X_n)=\mathbb{P}(X_{n+1}\in A|N_{n,A})\qquad\mbox{for all }A\in\mathcal{G},
\end{equation}
and
\begin{equation}\label{pure:result:main:eq2}
\mathbb{P}(X_{n+1}\in B|X_1,\ldots,X_n,\mathcal{G}_{n+1})=\mathbb{P}(X_{n+1}\in B|\mathcal{G}_{n+1})\qquad\mbox{for all }B\in\mathcal{X},
\end{equation}
then $(X_n)_{n\geq1}$ is an MVPS with reinforcement kernel $R(\cdot)=\nu(\cdot\mid\mathcal{G})$. The converse statement is also true.
\end{theorem}

\begin{example}[Dirichlet process]\label{pure:example:dirichlet_process}
The current example essentially repeats Proposition 3.3 in \cite{fortini2024}. Suppose in Theorem \ref{pure:result:main} that we have $\mathcal{G}=\mathcal{X}$, which is c.g. by assumption. Then \eqref{pure:result:main:eq2} is trivially true, so any exchangeable process $(X_n)_{n\geq1}$ satisfying
\[\mathbb{P}(X_{n+1}\in A|X_1,\ldots,X_n)=\mathbb{P}(X_{n+1}\in A|N_{n,A})\qquad\mbox{for all }A\in\mathcal{X},\]
is an MVPS with reinforcement kernel
\[R_x(\cdot)=\nu(\cdot\mid\mathcal{X})(x)=\delta_x(\cdot)\qquad\mbox{for }\nu\mbox{-a.e. }x,\]
and, following \cite{blackwell1973}, has a Dirichlet process prior.
\end{example}

\begin{example}[Mixture model]\label{pure:example:dominated}
In Theorem \ref{pure:result:main}, suppose that $\mathcal{G}=\sigma(D_1,D_2,\ldots)$, where $D_1,D_2,\ldots\in\mathcal{X}$ forms a countable partition of $\mathbb{X}$. Then \eqref{pure:result:main:eq1}-\eqref{pure:result:main:eq2} can be equivalently formulated as
\[\mathbb{P}(X_{n+1}\in D_k|X_1,\ldots,X_n)=\mathbb{P}(X_{n+1}\in D_k|N_{n,D_k})\qquad\mbox{for each }k=1,2,\ldots,\]
and, with a slight abuse of notation,
\[\mathbb{P}(X_{n+1}\in B|X_1,\ldots,X_n,X_{n+1}\in D_k)=\mathbb{P}(X_{n+1}\in B|X_{n+1}\in D_k)\qquad\mbox{for all }B\in\mathcal{X}.\]
In that case, $(X_n)_{n\geq1}$ is an MVPS with reinforcement kernel
\[R_x(\cdot)=\nu(\cdot\mid\sigma(D_1,D_2,\ldots))(x)=\sum_{k\geq1}\nu(\cdot\mid D_k)\cdot\mathbbm{1}_{D_k}(x)\qquad\mbox{for }\nu\mbox{-a.e. }x,\]
so that $R_x\ll\nu$ and, from Theorem \ref{model:theorem:dominated}, $(X_n)_{n\geq1}$ has a Dirichlet process mixture prior.
\end{example}

\begin{example}[Dirichlet invariant process]\label{pure:example:invariant}
Suppose that $\mathbb{X}=\mathbb{R}$. Define
\[\mathcal{G}:=\{A\in\mathcal{X}:-A=A\}\equiv\{A\cup(-A),A\in\mathcal{X}\},\]
%% the function $f(x)=-x$ is continuous, and so $-A\in\mathcal{X}$.
where $-A=\{-x,x\in A\}$. Let $(X_n)_{n\geq1}$ be an exchangeable process such that, for each $n=1,2,\ldots$ and every $A\in\mathcal{X}$,
\begin{equation}\label{pure:example:invariant:eq1}
\mathbb{P}(X_{n+1}\in A|X_1,\ldots,X_n)=\mathbb{P}(X_{n+1}\in A|N_{n,A}+N_{n,-A}),
\end{equation}
and
\begin{equation}\label{pure:example:invariant:eq2}
\mathbb{P}(X_{n+1}\in-A|X_1,\ldots,X_n)=\mathbb{P}(X_{n+1}\in A|X_1,\ldots,X_n).
\end{equation}
If we denote by $\nu$ the marginal distribution of $(X_n)_{n\geq1}$, we get $\nu(-A)=\nu(A)$ from \eqref{pure:example:invariant:eq2}, for every $A\in\mathcal{X}$. Moreover, \eqref{pure:example:invariant:eq1} implies \eqref{pure:result:main:eq1} when $A\in\mathcal{G}$, since $-A=A$.

Let $A_1,\ldots,A_n,A\in\mathcal{X}$ and $B\in\mathcal{G}$. Since $\nu$ is invariant with respect to $\mathcal{G}$, we have
\begingroup\allowdisplaybreaks
\begin{align*}
\mathbb{E}\bigl[\mathbbm{1}_B(X_{n+1})\cdot\mathbb{P}(X_{n+1}\in A|\mathcal{G}_{n+1})\bigr]&=\nu(A\cap B)=\frac{\nu(A\cap B)+\nu((-A)\cap B)}{2}\\
&=\mathbb{E}\Bigl[\mathbbm{1}_B(X_{n+1})\cdot\frac{1}{2}\bigl(\mathbbm{1}_{A}(X_{n+1})+\mathbbm{1}_{-A}(X_{n+1})\bigr)\Bigr];
\end{align*}
\endgroup
thus,
\[\mathbb{P}(X_{n+1}\in A|\mathcal{G}_{n+1})=\frac{\delta_{X_{n+1}}(A)+\delta_{-X_{n+1}}(A)}{2}\qquad\mbox{a.s.}\]
If, in addition, $A\cap(-A)=\emptyset$, then $\mathbb{P}(X_{n+1}\in A|\mathcal{G}_{n+1})=\frac{1}{2}\cdot\mathbbm{1}_{A\cup(-A)}(X_{n+1})$ a.s. and, from \eqref{pure:example:invariant:eq1} and \eqref{pure:example:invariant:eq2},
\begingroup\allowdisplaybreaks
\begin{align*}
\mathbb{E}[\mathbbm{1}_{A_1}(X_1)\cdots\mathbbm{1}_{A_n}(X_n)\mathbbm{1}_A&(X_{n+1})\mathbbm{1}_B(X_{n+1})]\\
&=\mathbb{E}\bigl[\mathbbm{1}_{A_1}(X_1)\cdots\mathbbm{1}_{A_n}(X_n)\cdot\mathbbm{P}(X_{n+1}\in A\cap B|N_{n,A\cap B}+N_{n,(-A)\cap B})\bigr]\\
&=\frac{1}{2}\cdot\mathbb{E}\bigl[\mathbbm{1}_{A_1}(X_1)\cdots\mathbbm{1}_{A_n}(X_n)\cdot\mathbbm{P}(X_{n+1}\in(A\cup(-A))\cap B|N_{n,(A\cup(-A))\cap B})\bigr]\\
&=\frac{1}{2}\cdot\mathbb{E}[\mathbbm{1}_{A_1}(X_1)\cdots\mathbbm{1}_{A_n}(X_n)\mathbbm{1}_{A\cup(-A)}(X_{n+1})\mathbbm{1}_B(X_{n+1})]\\
&=\mathbb{E}\bigl[\mathbb{P}(X_1\in A_1,\ldots,X_n\in A_n|\mathcal{G}_{n+1})\cdot\mathbb{P}(X_{n+1}\in A|\mathcal{G}_{n+1})\cdot\mathbbm{1}_B(X_{n+1})\bigr].
\end{align*}
\endgroup
By varying $A\in\mathcal{C}:=\{(-b,-a],(a,b]:a,b\in\mathbb{R}_+,a<b\}\cup\{\emptyset\}$ and using the fact that $\mathcal{C}$ forms a $\pi$-class generating $\mathcal{X}=\sigma(\mathcal{C})$, we can conclude that $X_{n+1}$ and $(X_1,\ldots,X_n)$ are conditionally independent given $\mathcal{G}_{n+1}$. Hence, $(X_n)_{n\geq1}$ satisfies the conditions of Theorem \ref{pure:result:main} and is therefore an exchangeable MVPS with reinforcement kernel
\[R_x(\cdot)=\nu(\cdot\mid\mathcal{G})(x)=\frac{\delta_x(\cdot)+\delta_{-x}(\cdot)}{2}\qquad\mbox{for }\nu\mbox{-a.e. }x.\]
This model has been studied by \cite{berti2023}, who show that $(X_n)_{n\geq1}$ has a so-called Dirichlet invariant process prior (for more information, see Examples 3 and 17 in \cite{berti2023}).
\end{example}

\subsection{k-color MVPS}\label{section:results:k-color}

In this section, we prove a generalization of Johnson's sufficientness postulate \eqref{intro:equation:johnson_sufficientness} in the context of  the $k$-color urn model from Example \ref{example:k_color}. Although some of the results can be deduced from previous sections, for completeness we present a separate proof in the spirit of \cite{zabell1982}, exploiting the discrete nature of the process.

Let $\mathbb{X}=\{1,\ldots,k\}$, and $\{D_1,\ldots,D_m\}\subseteq\mathcal{X}$ be a partition of $\mathbb{X}$, for some $1\leq m\leq k$. For each $j\in\mathbb{X}$, let us denote by $m(j)$ the index $l\in\{1,\ldots,m\}$ such that $j\in D_l$, and let us define
\[N_{0,l}:=0\qquad\mbox{and}\qquad N_{n,l}:=\sumN\mathbbm{1}_{D_l}(X_i),\qquad\mbox{for }n\geq1,l=1,\ldots,m.\]
If $(X_n)_{n\geq1}$ is an exchangeable MVPS with reinforcement kernel given by \eqref{example:k_color:representation}, then for each $j=1,\ldots,k$,
\begingroup\allowdisplaybreaks
\begin{align*}
\mathbb{P}(X_{n+1}=j|X_1,\ldots,X_n)&=\frac{\theta\nu(\{j\})+\sumN\nu(\{j\}|D_{m(j)})\cdot\mathbbm{1}_{D_{m(j)}}(X_i)}{\theta+n}=\frac{\theta\nu(\{j\})+\nu(\{j\}|D_{m(j)})\cdot N_{n,m(j)}}{\theta+n},
\end{align*}
\endgroup
that is, the probability of observing type $j$ on the $(n+1)$th trial depends on the sample $(X_1,\ldots,X_n)$ only by the number of times that each type in $D_{m(j)}$ has been observed in the past. This fact motivates the next proposition, which extends Johnson's sufficientness postulate \eqref{intro:equation:johnson_sufficientness} by considering a grouping of types according to $D_1,\ldots,D_m$.

For ease of notation, we will denote
\[\mathbb{P}_0(\cdot)=\mathbb{P}(X_1\in\cdot\;)\quad\mbox{and}\quad\mathbb{P}_n(\cdot)=\mathbb{P}(X_{n+1}\in\cdot\mid X_1,\ldots,X_n),\quad\mbox{for }n=1,2,\ldots\]
In addition, we will work under the assumption that any sequence of types has a positive probability of being observed, 
\begin{equation}\label{equation:k-color:assumption}\tag{A}
\mathbb{P}(X_1=x_1,\ldots,X_n=x_n)>0\qquad\mbox{for all }(x_1,\ldots,x_n)\in\mathbb{X}^n\mbox{ and }n=1,2,\ldots
\end{equation}

\begin{proposition}\label{result:k-color}
Let $D_1,\ldots,D_m$ form a partition of $\mathbb{X}=\{1,\ldots,k\}$, with $3\leq m\leq k$, and let $(X_n)_{n\geq1}$ be an exchangeable sequence of $\mathbb{X}$-valued random variables satisfying \eqref{equation:k-color:assumption}. Using the established notation, if for each $j=1,\ldots,k$ and $n=0,1,\ldots$, it holds
\[\mathbb{P}_n(\{j\})=f_{n,j}(N_{n,m(j)}),\]
for some functions $f_{n,j}:\{0,\ldots,n\}\rightarrow[0,1]$, then there exist $w_1,\ldots,w_k>0$ such that $(X_n)_{n\geq1}$ is an MVPS with constant $\bar{w}=\sum_{j=1}^kw_j$, base measure $\nu(\cdot)=\sum_{j=1}^k\frac{w_j}{\bar{w}}\delta_{j}(\cdot)$ and reinforcement kernel
\[R(\cdot)=\nu(\cdot\mid\sigma(D_1,\ldots,D_m)).\]
\end{proposition}

If $m=1$ in Proposition \ref{result:k-color}, it is easy to see that $(X_n)_{n\geq1}$ is i.i.d. with marginal distribution $\mathbb{P}_0$. When $m=2$, however, additional assumptions on $P_n$ are needed to obtain the stated characterization result (see the remark after the proof of Proposition \ref{result:k-color} in Section \ref{section:proof}). A different approach, following the work of Hill and coauthors \cite{hill1987}, is presented in the next proposition, which is more in line with Proposition \ref{result:k-color} and covers the case $m=2$. Hill et al.'s \cite{hill1987} original result is itself recovered by taking $D_j=\{j\}$, for $j=1,\ldots,k$.

\begin{proposition}\label{result:k-color:alternative}
Let $D_1,\ldots,D_m$ form a partition of $\mathbb{X}=\{1,\ldots,k\}$, with $1\leq m\leq k$, and let $(X_n)_{n\geq1}$ be an exchangeable, but not i.i.d. sequence of $\mathbb{X}$-valued random variables satisfying \eqref{equation:k-color:assumption}. Fix $w_1,\ldots,w_k>0$, and define
\[\bar{w}:=\sum_{j=1}^kw_j\qquad\mbox{and}\qquad\bar{w}_l:=\sum_{j\in D_l}w_j,\quad\mbox{for }l=1,\ldots,m.\]
Using the established notation, if for each $j=1,\ldots,k$ and $n=0,1,\ldots$, it holds
\[\mathbb{P}_n(\{j\})=f_j\biggl(\frac{\bar{w}_{m(j)}+N_{n,m(j)}}{\bar{w}+n}\biggr),\]
for some functions $f_j:[0,1]\rightarrow[0,1]$, then $(X_n)_{n\geq1}$ is an MVPS with constant $\bar{w}$, base measure $\nu(\cdot)=\sum_{j=1}^k\frac{w_j}{\bar{w}}\delta_{j}(\cdot)$ and reinforcement kernel
\[R(\cdot)=\nu(\cdot\mid\sigma(D_1,\ldots,D_m)).\]
\end{proposition}

\section{Proofs}\label{section:proof}

\begin{proof}[Proof of Theorem \ref{result:main}] Theorem \ref{result:main} will be proven in several steps. We start by deriving some relations between $\nu$, $R$ and $(a_n)_{n\geq1}$ as a consequence of exchangeability.\\

\noindent\textit{Step 1.} Let $A,B\in\mathcal{X}$. It follows from \eqref{equation:predictive_distributions} and $\mathbb{P}(X_1\in dx)=\nu(dx)$ that
\begingroup\allowdisplaybreaks
\begin{align*}
\mathbb{P}(X_1\in A,X_2\in B)&=\int_A\mathbb{P}(X_2\in B|X_1=x)\mathbb{P}(X_1\in dx)\\
&=\int_A\bigl\{(1-a_1)\nu(B)+a_1R_x(B)\bigr\}\nu(dx)=(1-a_1)\nu(A)\nu(B)+a_1\int_AR_x(B)\nu(dx).
\end{align*}
\endgroup
Since $(X_1,X_2)\overset{d}{=}(X_2,X_1)$ by exchangeability, then
\begingroup\allowdisplaybreaks
\begin{align*}
(1-a_1)\nu(A)\nu(B)+a_1\int_AR_x(B)\nu(dx)&=\mathbb{P}(X_1\in A,X_2\in B)\\
&=\mathbb{P}(X_1\in B,X_2\in A)=(1-a_1)\nu(A)\nu(B)+a_1\int_BR_x(A)\nu(dx);
\end{align*}
\endgroup
therefore, $\int_{x\in A}\int_{y\in B}R_x(dy)\nu(dx)=\int_{x\in B}\int_{y\in A}R_x(dy)\nu(dx)$ or, as measures on $\mathbb{X}^2$,
\begin{equation}\label{proof:identity1}
R_x(dy)\nu(dx)=R_y(dx)\nu(dy).
\end{equation}

Next, we have for $\nu$-a.e. $x$ that
\begingroup\allowdisplaybreaks
\begin{align*}
\mathbb{P}(X_3\in B|X_1=x)&\,=\int_\mathbb{X}\mathbb{P}(X_3\in B|X_1=x,X_2=y)\mathbb{P}(X_2\in dy|X_1=x)\\
&\,=\int_\mathbb{X}\Bigl\{(1-a_2)\nu(B)+\frac{a_2}{2}R_x(B)+\frac{a_2}{2}R_y(B)\Bigr\}\bigl((1-a_1)\nu(dy)+a_1R_x(dy)\bigr)\\
&\,=(1-a_2)\nu(B)+\frac{a_2}{2}R_x(B)+\frac{(1-a_1)a_2}{2}\int_\mathbb{X}R_y(B)\nu(dy)+\frac{a_1a_2}{2}\int_\mathbb{X}R_y(B)R_x(dy)\\
&\overset{(a)}{=}(1-a_2)\nu(B)+\frac{a_2}{2}R_x(B)+\frac{(1-a_1)a_2}{2}\nu(B)+\frac{a_1a_2}{2}\int_\mathbb{X}R_y(B)R_x(dy),
\end{align*}
\endgroup
for every $B\in\mathcal{X}$, where $(a)$ follows from \eqref{proof:identity1} as $\int_\mathbb{X}R_y(B)\nu(dy)=\nu(B)$. Since $(X_1,X_2)\overset{d}{=}(X_1,X_3)$ by exchangeability, we get for $\nu$-a.e. $x$ that
\begingroup\allowdisplaybreaks
\begin{align*}
(1-a_1)\nu(B)+a_1 R_x(B)&=\mathbb{P}(X_2\in B|X_1=x)\\
&=\mathbb{P}(X_3\in B|X_1=x)\\
&=(1-a_2)\nu(B)+\frac{a_2}{2}R_x(B)+\frac{(1-a_1)a_2}{2}\nu(B)+\frac{a_1a_2}{2}\int_\mathbb{X}R_y(B)R_x(dy);
\end{align*}
\endgroup
thus, rearranging the terms,
\[\int_\mathbb{X}R_y(B)R_x(dy)=\Bigl(1-\frac{a_2-2a_1+a_1a_2}{a_1a_2}\Bigr)R_x(B)+\frac{a_2-2a_1+a_1a_2}{a_1a_2}\nu(B).\]
Using that $\mathcal{X}$ is countably generated, we obtain for $\nu$-a.e. $x$ that, as measures on $\mathbb{X}$,
\begin{equation}\label{proof:identity2}
\int_{y\in\mathbb{X}}R_y(dz)R_x(dy)=(1-c^*)R_x(dz)+c^*\nu(dz),
\end{equation}
where $c^*:=(a_2-2a_1+a_1a_2)/(a_1a_2)$.
%% Using a monotone class argument.

Let $n\geq2$. It follows from \eqref{equation:predictive_distributions} that, for a.e. $(x_1,\ldots,x_{n-1},x)$ in $\mathbb{X}^n$ with respect to the marginal distribution $\mathbb{P}_{(X_1,\ldots,X_n)}$ of $(X_1,\ldots,X_n)$ on $\mathcal{X}^n$,
\begingroup\allowdisplaybreaks
\begin{align}\label{proof:a.e.set}
\begin{aligned}
\mathbb{P}(X_{n+1}\in B|X_1=x_1,\ldots,&X_{n-1}=x_{n-1},X_n=x)=(1-a_n)\nu(dy)+\frac{a_n}{n}\sum_{i=1}^{n-1}R_{x_i}(B)+\frac{a_n}{n}R_x(B),
\end{aligned}
\end{align}
\endgroup
for every $B\in\mathcal{X}$. On the other hand, disintegrating $\mathbb{P}_{(X_1,\ldots,X_n)}$ repeatedly via \eqref{equation:predictive_distributions}, we obtain (see \eqref{proof:abs_cont:disintegration})
\[\mathbb{P}(X_1\in dx_1,\ldots,X_n\in dx_n)\geq\prod_{i=1}^{n-1}(1-a_i)\nu(dx_1)\cdots\nu(dx_n).\]
Thus, \eqref{proof:a.e.set} remains true for $(\nu\times\cdots\times\nu)$-a.e. $(x_1,\ldots,x_{n-1},x)$ in $\mathbb{X}^n$, so integrating \eqref{proof:a.e.set} with respect to $\nu(dx_1),\ldots,\nu(dx_{n-1})$, we get that, for $\nu$-a.e. $x$,
%%%% To show that the equation holds for $\nu$-a.e. $x$, given what is written in the footnote, consider $C=\{(x,y)\in\mathbb{X}^2:f(x,y)=g(x,y)\}$ such that $(\nu\times\nu)(C)=1$. It follows for every $A\in\mathcal{X}$ that \[\int_A\int_\mathbb{X}f(x,y)\nu(dy)\nu(dx)=\int_{(A\times\mathbb{X})\cap C}f(x,y)\(\nu\times\nu)(dx,dy)=\int_A\int_\mathbb{X}g(x,y)\nu(dy)\nu(dx),\] so $\int_\mathbb{X}f(x,y)\nu(dy)=\int_\mathbb{X}g(x,y)\nu(dy)$ for $\nu$-a.e. $x$.
\begingroup\allowdisplaybreaks
\begin{align}
\int_\mathbb{X}\cdots\int_\mathbb{X}\mathbb{P}(X_{n+1}\in B|&X_1=x_1,\ldots,X_{n-1}=x_{n-1},X_n=x)\nu(dx_{n-1})\cdots\nu(dx_1) \nonumber\\
&=(1-a_n)\nu(B)+\frac{a_n}{n}\sum_{i=1}^{n-1}\nu(B)+\frac{a_n}{n}R_x(B)\nonumber\\
&=(1-a_n)\nu(B)+\frac{(n-1)a_n}{n}\nu(B)+\frac{a_n}{n}R_x(B),\label{proof:equation:partial_result1}
\end{align}
\endgroup
where we have used that $\int_\mathbb{X}R_{x_i}(B)\nu(dx_i)=\nu(B)$ from \eqref{proof:identity1}, for $i=1,\ldots,n-1$.

On the other hand, for $(\nu\times\cdots\times\nu)$-a.e. $(x_1,\ldots,x_{n-1},x)$ in $\mathbb{X}^n$,
\begingroup\allowdisplaybreaks
\begin{align}
\mathbb{P}(X_{n+2}&\in B|X_1=x_1,\ldots,X_{n-1}=x_{n-1},X_n=x) \nonumber\\
&=\int_\mathbb{X}\mathbb{P}(X_{n+2}\in B|X_1=x_1,\ldots,X_n=x,X_{n+1}=y)\mathbb{P}(X_{n+1}\in dy|X_1=x_1,\ldots,X_n=x) \nonumber\\
&\,\begin{aligned}=\int_\mathbb{X}\Bigl\{(1-a_{n+1})\nu(B)&+\frac{a_{n+1}}{n+1}\sum_{i=1}^{n-1}R_{x_i}(B)+\frac{a_{n+1}}{n+1}R_x(B)\\
&+\frac{a_{n+1}}{n+1}R_y(B)\Bigr\}\Bigl((1-a_n)\nu(dy)+\frac{a_n}{n}\sum_{i=1}^{n-1} R_{x_i}(dy)+\frac{a_n}{n}R_x(dy)\Bigr)\end{aligned} \nonumber\\
&\,\begin{aligned}=(1-a_{n+1})\nu(B)&+\frac{a_{n+1}}{n+1}\sum_{i=1}^{n-1}R_{x_i}(B)+\frac{a_{n+1}}{n+1}R_x(B)+\frac{(1-a_n)a_{n+1}}{n+1}\int_\mathbb{X}R_y(B)\nu(dy)\\
&+\frac{a_na_{n+1}}{n(n+1)}\sum_{i=1}^{n-1}\int_\mathbb{X}R_y(B)R_{x_i}(dy)+\frac{a_na_{n+1}}{n(n+1)}\int_\mathbb{X}R_y(B)R_x(dy)\end{aligned} \nonumber\\
&\begin{aligned}\overset{(a)}{=}(1-a_{n+1})&\nu(B)+\frac{a_{n+1}}{n+1}\sum_{i=1}^{n-1}R_{x_i}(B)+\frac{a_{n+1}}{n+1}R_x(B)+\frac{(1-a_n)a_{n+1}}{n+1}\nu(B)\\
&+\frac{a_na_{n+1}}{n(n+1)}\sum_{i=1}^{n-1}\int_\mathbb{X}R_y(B)R_{x_i}(dy)+(1-c^*)\frac{a_na_{n+1}}{n(n+1)}R_x(B)+c^*\frac{a_na_{n+1}}{n(n+1)}\nu(B),\end{aligned}\label{proof:equation:partial_result2a}
\end{align}
\endgroup
for every $B\in\mathcal{X}$, where $(a)$ follows from \eqref{proof:identity1} and \eqref{proof:identity2}. Note that, for each $i=1,\ldots,n-1$, the following identity is true from \eqref{proof:identity1},
\begingroup\allowdisplaybreaks
\begin{align*}
\int_{x_i\in\mathbb{X}}\int_{y\in\mathbb{X}}R_y(B)R_{x_i}(dy)\nu(dx_i)&=\int_{x_i\in\mathbb{X}}\int_{y\in\mathbb{X}}R_y(B)R_y(dx_i)\nu(dy)\\
&=\int_{y\in\mathbb{X}}R_y(B)R_y(\mathbb{X})\nu(dy)=\int_{y\in\mathbb{X}}R_y(B)\nu(dy)=\nu(B).
\end{align*}
\endgroup
Thus, integrating \eqref{proof:equation:partial_result2a} with respect to $\nu(dx_1),\ldots,\nu(dx_{n-1})$, we get that, for $\nu$-a.e. $x$,
\begingroup\allowdisplaybreaks
\begin{align}
\int_{x_1\in\mathbb{X}}\cdots&\int_{x_{n-1}\in\mathbb{X}}\mathbb{P}(X_{n+2}\in B|X_1=x_1,\ldots,X_{n-1}=x_{n-1},X_n=x)\nu(dx_{n-1})\cdots\nu(dx_1) \nonumber\\
&\begin{aligned}=(1-a_{n+1})\nu(B)&+\frac{a_{n+1}}{n+1}\sum_{i=1}^{n-1}\nu(B)+\frac{a_{n+1}}{n+1}R_x(B)+\frac{(1-a_n)a_{n+1}}{n+1}\nu(B)\\
&+\frac{a_na_{n+1}}{n(n+1)}\sum_{i=1}^{n-1}\nu(B)+(1-c^*)\frac{a_na_{n+1}}{n(n+1)}R_x(B)+c^*\frac{a_na_{n+1}}{n(n+1)}\nu(B)\end{aligned} \nonumber\\
&\begin{aligned}=(1-a_{n+1})\nu(B)&+\frac{(n-1)a_{n+1}}{n+1}\nu(B)+\frac{a_{n+1}}{n+1}R_x(B)+\frac{(1-a_n)a_{n+1}}{n+1}\nu(B)\\
&+\frac{(n-1)a_na_{n+1}}{n(n+1)}\nu(B)+(1-c^*)\frac{a_na_{n+1}}{n(n+1)}R_x(B)+c^*\frac{a_na_{n+1}}{n(n+1)}\nu(B).\end{aligned}\label{proof:equation:partial_result2b}
\end{align}
\endgroup
Since $(X_1,\ldots,X_n,X_{n+1})\overset{d}{=}(X_1,\ldots,X_n,X_{n+2})$ by exchangeability, we can equate  \eqref{proof:equation:partial_result1} and \eqref{proof:equation:partial_result2b} to obtain, for $\nu$-a.e. $x$,
\begingroup\allowdisplaybreaks
\begin{align*}
(1-a_n)\nu(B)&+\frac{(n-1)a_n}{n}\nu(B)+\frac{a_n}{n}R_x(B)\\
&=\int_\mathbb{X}\cdots\int_\mathbb{X}\mathbb{P}(X_{n+1}\in B|X_1=x_1,\ldots,X_{n-1}=x_{n-1},X_n=x)\nu(dx_{n-1})\cdots\nu(dx_1)\\
&=\int_\mathbb{X}\cdots\int_\mathbb{X}\mathbb{P}(X_{n+2}\in B|X_1=x_1,\ldots,X_{n-1}=x_{n-1},X_n=x)\nu(dx_{n-1})\cdots\nu(dx_1)\\
&\,\begin{aligned}=(1-a_{n+1})\nu(B)&+\frac{(n-1)a_{n+1}}{n+1}\nu(B)+\frac{a_{n+1}}{n+1}R_x(B)+\frac{(1-a_n)a_{n+1}}{n+1}\nu(B)\\
&+\frac{(n-1)a_na_{n+1}}{n(n+1)}\nu(B)+(1-c^*)\frac{a_na_{n+1}}{n(n+1)}R_x(B)+c^*\frac{a_na_{n+1}}{n(n+1)}\nu(B),\end{aligned}
\end{align*}
\endgroup
for every $B\in\mathcal{X}$; therefore, after some simple algebra, 
\[\bigl(na_{n+1}-(n+1)a_n+(1-c^*)a_na_{n+1}\bigr)\nu(B)=\bigl(na_{n+1}-(n+1)a_n+(1-c^*)a_na_{n+1}\bigr)R_x(B).
\]
Using that $\mathcal{X}$ is countably generated, we get for $\nu$-a.e. $x$ that, as measures on $\mathbb{X}$,
\begin{equation}\label{proof:equation:constants}
\bigl(na_{n+1}-(n+1)a_n+(1-c^*)a_na_{n+1}\bigr)\nu(dy)=\bigl(na_{n+1}-(n+1)a_n+(1-c^*)a_na_{n+1}\bigr)R_x(dy).
\end{equation}~ 

\noindent\textit{Step 2.} We proceed by showing that, for different specifications of $R$, equations \eqref{proof:identity2} and \eqref{proof:equation:constants} together imply the prescribed form of the coefficients $a_n$ and, in particular, that $c^*=0$. To that end, let us consider for each $x\in\mathbb{X}$ the Lebesgue decomposition of the reinforcement kernel
\[R_x=R_x^\perp+R_x^a,\]
where $R_x^\perp$ and $R_x^a$ are substochastic measures on $\mathbb{X}$ such that $R_x^\perp\perp\nu$ for some $S_x\in\mathcal{X}$ with $R_x^\perp(S_x^c)=0=\nu(S_x)$, and $R_x^a\ll\nu$. It follows from Theorem 1 in \cite{lange1973} that the mappings $x\mapsto R_x^\perp$ and $x\mapsto R_x^a$ are measurable, hence the following sets are well-defined in $\mathcal{X}$,
\[D:=\{x\in\mathbb{X}:R_x\neq\nu\}\qquad\mbox{and}\qquad F:=\{x\in\mathbb{X}:R_x^\perp(S_x)>0\}.\]

\noindent\textit{Case 1: $\nu(D)>0$.}

By definition, for every $x\in D$, there exists $B_x\in\mathcal{X}$ such that $\nu(B_x)\neq R_x(B_x)$. It follows from \eqref{proof:equation:constants} that, for $\nu$-a.e. $x$ in $D$,
\[\bigl(na_{n+1}-(n+1)a_n+(1-c^*)a_na_{n+1}\bigr)\nu(B_x)=\bigl(na_{n+1}-(n+1)a_n+(1-c^*)a_na_{n+1}\bigr)R_x(B_x),\]
which is true if and only if
\[na_{n+1}-(n+1)a_n+(1-c^*)a_na_{n+1}=0,\]
or, equivalently,
\[\frac{a_{n+1}}{n+1}=\frac{a_n}{n+(1-c^*)a_n}.\]
Using the substitution $b_n=\frac{n}{a_n}$, we obtain $b_{n+1}=b_n+(1-c^*)$; thus, $b_n=b_1+(n-1)(1-c^*)$ for each $n=1,2,\ldots$, and so
\begin{equation}\label{proof:weights}
a_n=\frac{n}{(n-1)(1-c^*)+\frac{1}{a_1}}\qquad\mbox{for }n=1,2,\ldots.
\end{equation}
If $c^*=1$, then $a_n=na_1$, which leads to a contradiction since, by assumption, $0<a_n<1$ for all $n=1,2,\ldots$; otherwise,
\[\limN a_n=\limN\frac{1}{(1-c^*)-\frac{1-c^*}{n}+\frac{1}{a_1n}}=\frac{1}{1-c^*},\]
implying that $a_n\notin(0,1)$ for $n$ large enough, absurd, unless $c^*\leq0$.\\

\noindent\textit{Sub-Case 1a: $\nu(D)>0$ and $\nu(F)>0$.}

Let $G\in\mathcal{X}$ be the $\nu$-a.s. set in \eqref{proof:identity2}. Define
\[G_0:=\{x\in G:R_x(G)=1\}\quad\mbox{and}\quad G_n:=\{x\in G_{n-1}:R_x(G_{n-1})=1\},\quad\mbox{for }n\geq1.\]
It follows from \eqref{proof:identity1} that
\[1=\int_\mathbb{X}R_x(\mathbb{X})\nu(dx)=\int_GR_x(\mathbb{X})\nu(dx)=\int_\mathbb{X}R_x(G)\nu(dx),\]
which implies that $R_x(G)=1$ for $\nu$-a.e. $x$, and so $\nu(G_0)=1$. Proceeding by induction, we obtain $\nu(G_n)=1$ for all $n\geq1$; thus, letting $G^*:=\bigcap_{n=0}^\infty G_n$, we have $\nu(G^*)=1$ and $R_x(G_n)=1$, for all $x\in G^*$ and $n\geq1$. As a result, for every $x\in G^*$, we get $R_x(G^*)=1$ and, since $G^*\subseteq G$,
\begin{equation}\label{proof:identity2.1}
\int_{y\in\mathbb{X}}R_y(dz)R_x(dy)=(1-c^*)R_x(dz)+c^*\nu(dz).
\end{equation}

Recall that $F=\{x\in\mathbb{X}:R_x^\perp(S_x)>0\}$. Let us define
\[\beta:=\sup_{x\in F\cap G^*}R_x^\perp(S_x).\]
Assuming that $\nu(F)>0$, we have $\beta>0$. Fix $\epsilon\in(0,\beta)$. From the properties of suprema, there exists $x\in F\cap G^*$ such that $R_x^\perp(S_x)\geq\beta-\epsilon$. Moreover, $\nu(S_x)=0$, which implies that $R_y^a(S_x)=0$ for all $y\in\mathbb{X}$. It follows that, first
\[(1-c^*)R_x(S_x)+c^*\nu(S_x)=(1-c^*)R_x^\perp(S_x)\geq(1-c^*)(\beta-\epsilon),\]
since $c^*\leq0$ from the previous part, and second
\[\int_\mathbb{X}R_y(S_x)R_x(dy)=\int_\mathbb{X}R_y^\perp(S_x)R_x(dy)\overset{(a)}{\leq}\int_\mathbb{X}R_y^\perp(S_y)R_x(dy)\overset{(b)}{=}\int_{F\cap G^*}R_y^\perp(S_y)R_x(dy)\leq\beta,\]
where we have used in $(a)$ that $R_y^\perp(S_y^c)=0$, so $R_y^\perp(S_x)=R_y^\perp(S_x\cap S_y)\leq R_y^\perp(S_y)$, for $y\in\mathbb{X}$; and in $(b)$ that $(i)$ $R_y^\perp(S_y)=0$ for $y\in F^c$, by construction, and $(ii)$ $R_x(G^*)=1$, since $x\in G^*$. Therefore, by \eqref{proof:identity2.1},
\[\beta\geq\int_\mathbb{X}R_y(S_x)R_x(dy)=(1-c^*)R_x(S_x)+c^*\nu(S_x)\geq(1-c^*)(\beta-\epsilon).\]
Letting $\epsilon\downarrow0$, we get $1-c^*\leq\frac{\beta}{\beta-\epsilon}\rightarrow 1$, so $c^*\geq 0$. As a consequence, $c^*=0$, which implies from \eqref{proof:weights} that
\[a_n=\frac{n}{n+\frac{1}{a_1}-1}.\]
Denoting $\theta:=a_1^{-1}-1$, we obtain $\theta>0$ and $a_n=n/(n+\theta)$, for each $n=1,2,\ldots$, so
\[\mathbb{P}(X_{n+1}\in dy\mid X_1,\ldots,X_n)=\frac{\theta}{\theta+n}\nu(dy)+\frac{n}{\theta+n}\cdot\frac{1}{n}\sumN R_{X_i}(dy)=\frac{\theta\nu(dy)+\sumN R_{X_i}(dy)}{\theta+n}\qquad\mbox{a.s.},\]
that is, $(X_n)_{n\geq1}$ is an MVPS with parameters $(\theta,\nu,R)$.\\

\noindent\textit{Sub-Case 1b: $\nu(D)>0$ and $\nu(F)=0$.}

If $\nu(F)=0$ instead, then $R_x\ll\nu$ for $\nu$-a.e. $x$, so $R_x(dy)\nu(dx)\ll\nu(dx)\nu(dy)$ and there exists a jointly measurable function $r:\mathbb{X}^2\rightarrow\mathbb{R}_+$ such that
\begin{equation}\label{proof:abs_cont:equation}
R_x(dy)\nu(dx)=r(x,y)\nu(dx)\nu(dy).
\end{equation}
In that case, \eqref{proof:identity1} implies that, as measures on $\mathbb{X}^2$,
\begin{equation}\label{proof:abs_cont:identity1}
r(x,y)\nu(dx)\nu(dy)=r(y,x)\nu(dx)\nu(dy).
\end{equation}

Let $A,B,C\in\mathcal{X}$. It follows from the form of the predictive distributions \eqref{equation:predictive_distributions} that
\begingroup\allowdisplaybreaks
\begin{align}
\mathbb{P}(X_1&\in A,X_2\in B,X_3\in C) \nonumber\\
&=\int_A\int_B\mathbb{P}(X_3\in C|X_1=x,X_2=y)\mathbb{P}(X_2\in dy|X_1=x)\nu(dx) \nonumber\\
&=\int_A\int_B\Bigl\{(1-a_2)\nu(C)+\frac{a_2}{2}R_x(C)+\frac{a_2}{2}R_y(C)\Bigr\}\bigl((1-a_1)\nu(dy)+a_1R_x(dy)\bigr)\nu(dx) \nonumber\\
&\begin{aligned}=\int_A\Bigl\{(1-a_1)&(1-a_2)\nu(B)\nu(C)+a_1(1-a_2)\nu(C)R_x(B)+\frac{(1-a_1)a_2}{2}\nu(B)R_x(C)\\
&+\frac{a_1a_2}{2}R_x(B)R_x(C)+\frac{(1-a_1)a_2}{2}\int_BR_y(C)\nu(dy)+\frac{a_1a_2}{2}\int_BR_y(C)R_x(dy)\Bigr\}\nu(dx)\end{aligned} \nonumber\\
&\begin{aligned}=(1-a_1)(1-a_2)\nu(A)&\nu(B)\nu(C)+a_1(1-a_2)\nu(C)\int_AR_x(B)\nu(dx)\\
&+\frac{(1-a_1)a_2}{2}\nu(B)\int_AR_x(C)\nu(dx)+\frac{a_1a_2}{2}\int_AR_x(B)R_x(C)\nu(dx)\\
&+\frac{(1-a_1)a_2}{2}\nu(A)\int_BR_y(C)\nu(dy)+\frac{a_1a_2}{2}\int_A\int_BR_y(C)R_x(dy)\nu(dx).\end{aligned}\label{proof:abs_cont:disintegration}
\end{align}
\endgroup
By exchangeability,
\[\mathbb{P}(X_1\in A,X_2\in B,X_3\in C)=\mathbb{P}(X_1\in A,X_2\in C,X_3\in B),\]
so from \eqref{proof:abs_cont:disintegration} we obtain, after some simplification, 
\begingroup\allowdisplaybreaks
\begin{align*}
&\begin{aligned}a_1(1-a_2)\nu(C)\int_AR_x(B)\nu(dx)&+\frac{(1-a_1)a_2}{2}\nu(B)\int_AR_x(C)\nu(dx)\\
&+\frac{(1-a_1)a_2}{2}\nu(A)\int_BR_y(C)\nu(dy)+\frac{a_1a_2}{2}\int_A\int_BR_y(C)R_x(dy)\nu(dx)\end{aligned}\\
&\begin{aligned}=a_1(1-a_2)\nu(B)\int_AR_x(C)\nu(dx)&+\frac{(1-a_1)a_2}{2}\nu(C)\int_AR_x(B)\nu(dx)\\
&+\frac{(1-a_1)a_2}{2}\nu(A)\int_CR_y(B)\nu(dy)+\frac{a_1a_2}{2}\int_A\int_CR_y(B)R_x(dy)\nu(dx).\end{aligned}
\end{align*}
\endgroup
Using that $\int_BR_y(C)\nu(dy)=\int_CR_y(B)\nu(dy)$ from \eqref{proof:identity1}, we further get
\begingroup\allowdisplaybreaks
\begin{align*}
a_1&(1-a_2)\nu(C)\int_AR_x(B)\nu(dx)+\frac{(1-a_1)a_2}{2}\nu(B)\int_AR_x(C)\nu(dx)+\frac{a_1a_2}{2}\int_A\int_BR_y(C)R_x(dy)\nu(dx)\\
&=a_1(1-a_2)\nu(B)\int_AR_x(C)\nu(dx)+\frac{(1-a_1)a_2}{2}\nu(C)\int_AR_x(B)\nu(dx)+\frac{a_1a_2}{2}\int_A\int_CR_y(B)R_x(dy)\nu(dx),
\end{align*}
\endgroup
which after some simple algebra becomes
\begingroup\allowdisplaybreaks
\begin{align*}
\int_A\int_BR_y(C)R_x(dy)\nu(dx)-c^*\cdot\nu(C)&\int_AR_x(B)\nu(dx)\\
&=\int_A\int_CR_y(B)R_x(dy)\nu(dx)-c^*\cdot\nu(B)\int_AR_x(C)\nu(dx),
\end{align*}
\endgroup
where $c^*=\frac{a_2-2a_1+a_1a_2}{a_1a_2}\leq 0$ from before. In terms of the representation \eqref{proof:abs_cont:equation}, the above becomes
\begingroup\allowdisplaybreaks
\begin{align*}
\int_A\int_B\int_C\bigl(r(x,y)r(y,z)-c^*r(x,y)\bigr)\nu(dz)\nu(dy)\nu(dx)&=\int_A\int_B\int_C\bigl(r(x,z)r(z,y)-c^*r(x,z)\bigr)\nu(dz)\nu(dy)\nu(dx)\\
&=\int_A\int_B\int_C\bigl(r(x,z)r(y,z)-c^*r(x,z)\bigr)\nu(dz)\nu(dy)\nu(dx),
\end{align*}
\endgroup
where the last equation follows from \eqref{proof:abs_cont:identity1}. As a result, for $(\nu\times\nu\times\nu)$-a.e. $(x,y,z)$ in $\mathbb{X}^3$,
\[r(x,y)(r(y,z)-c^*)=r(x,z)(r(y,z)-c^*).\]

Suppose by contradiction that $c^*<0$. In such a case, after canceling the common term above, we get that, for $(\nu\times\nu\times\nu)$-a.e. $(x,y,z)$ in $\mathbb{X}^3$,
\begin{equation}\label{proof:abs_cont:identity2}
r(x,y)=r(x,z).
\end{equation}
Together \eqref{proof:abs_cont:identity1} and \eqref{proof:abs_cont:identity2} imply that, for $(\nu\times\nu\times\nu)$-a.e. $(x,y,z)$ in $\mathbb{X}^3$,
\begin{equation}\label{proof:abs_cont:identity3}
r(x,y)r(y,z)=r(x,z)r(z,y).
\end{equation}
Recall that, for every $x\in D$, there exists $B_x\in\mathcal{X}$ such that $\nu(B_x)\neq R_x(B)$. It follows from \eqref{proof:abs_cont:identity3} that, for $\nu$-a.e. $x$ in $D$,
\begingroup\allowdisplaybreaks
\begin{align*}
\int_\mathbb{X}R_y(B_x)R_x(dy)&=\int_{y\in\mathbb{X}}\int_{z\in B_x}r(x,y)r(y,z)\nu(dz)\nu(dy)\\
&=\int_{y\in B_x}\int_{z\in\mathbb{X}}r(x,y)r(y,z)\nu(dz)\nu(dy)=\int_{B_x}R_y(\mathbb{X})R_x(dy)=R_x(B_x).
\end{align*}
\endgroup
On the other hand, by \eqref{proof:identity2}, for $\nu$-a.e. $x$ in $D$,
\[\int_\mathbb{X}R_y(B_x)R_x(dy)=(1-c^*)R_x(B_x)+c^*\nu(B_x);\]
therefore, $R_x(B_x)=\nu(B_x)$, absurd. As a result, $c^*=0$, which implies as in the previous sub-case that $(X_n)_{n\geq1}$ is an MVPS with parameters $(a^{-1}-1,\nu,R)$.\\

\noindent\textit{Case 2: $\nu(D)=0$.}

If $\nu(D)=0$, then $R_x\equiv\nu$ for $\nu$-a.e. $x$. Since $\mathbb{P}(X_1\in\cdot\,)=\nu(\cdot)$, it follows for every $B_1,B_2\in\mathcal{X}$ that
\begingroup\allowdisplaybreaks
\begin{align*}
\mathbb{P}(X_1\in B_1,X_2\in B_2)&=\int_{B_1}\mathbb{P}(X_2\in B_2|X_1=x)\nu(dx)\\
&=\int_{B_1}\bigl\{(1-a_1)\nu(B_2)+a_1R_x(B_2)\bigr\}\nu(dx)=\nu(B_1)\nu(B_2);
\end{align*}
\endgroup
thus, $\mathbb{P}_{(X_1,X_2)}(\cdot)=(\nu\times\nu)(\cdot)$. By induction, for every $B_1,\ldots,B_n\in\mathcal{X}$,
\[\mathbb{P}(X_1\in B_1,\ldots,X_n\in B_n)=\prod_{i=1}^n\nu(B_i);\]
therefore, $(X_n)_{n\geq1}$ is i.i.d. with marginal distribution $\nu$, and thus trivially an MVPS with parameters $(a^{-1}-1,\nu,R)$. This concludes the proof of Theorem \ref{result:main}.
\end{proof}

\begin{proof}[Proof of Theorem \ref{pure:result:main}]
Suppose that $(X_n)_{n\geq1}$ satisfies \eqref{pure:result:main:eq1}-\eqref{pure:result:main:eq2}. By assumption, $\mathcal{G}$ is c.g. under $\nu$, so there exists (see \cite[p.649]{berti2007}) an a.e. proper regular version $R$ of $\nu(\cdot\mid\mathcal{G})$, in the sense that, for some $F\in\mathcal{G}$ such that $\nu(F)=1$, it holds
\[R_x(A)=\delta_x(A)\qquad\mbox{for all }A\in\mathcal{G}\mbox{ and }x\in F.\]
Moreover, since $(X_n)_{n\geq1}$ is identically distributed with marginal distribution $\nu$, we have $\mathbb{P}(X_1\in F,\ldots,X_n\in F)=1$, for each $n=1,2,\ldots$, which implies that, on a set of probability one,
\begin{equation}\label{proof:pure:eq1}
\sumN\delta_{X_i}(A)=\sumN R_{X_i}(A)\qquad\mbox{for all }A\in\mathcal{G}.
\end{equation}

Let $B_1,\ldots,B_k\in\mathcal{G}$ form a partition of $\mathbb{X}$ with $k>2$ and $\#\{j:\nu(B_j)>0\}>2$. It follows from \eqref{pure:result:main:eq1} that the sequence $(X_n)_{n\geq1}$ on the level of $\{B_1,\ldots,B_k\}$ satisfies Johnson's sufficientness postulate \eqref{intro:equation:johnson_sufficientness}, so from Theorem 2.1 and Corollary 2.1 in \cite{zabell1982}, there exist $\alpha_{B_1},\ldots,\alpha_{B_k}\geq0$ such that $\sum_{i=1}^k\alpha_{B_i}>0$ and, for each $j=1,\ldots,k$,
\[\mathbb{P}(X_{n+1}\in B_j|X_1,\ldots,X_n)=\frac{\alpha_{B_j}+N_{n,B_j}}{\sum_{i=1}^k\alpha_{B_i}+n}\qquad\mbox{a.s.}\]
In particular, $\mathbb{P}(X_1\in B_j)=\alpha_{B_j}\bigl/\sum_{i=1}^k\alpha_{B_i}$, from where $a_{B_j}=\nu(B_j)\sum_{i=1}^k\alpha_{B_i}$. Moreover, $\theta:=\sum_{i=1}^k\alpha_{B_i}$ does not depend on the particular partition chosen. Indeed, let $\{C_1,\ldots,C_m\}\subseteq\mathcal{G}$ be another partition of $\mathbb{X}$ with $m>2$ and $\#\{l:\nu(C_l)>0\}>2$. If $\{C_1,\ldots,C_m\}$ is finer than $\{B_1,\ldots,B_k\}$ and, say, $B_j=\bigcup_{l=1}^LC_{j_l}$, where $\nu(B_j)>0$, we obtain, using the results in \cite{zabell1982},
\begingroup\allowdisplaybreaks
\begin{align*}
\frac{\theta\nu(B_j)+N_{n,B_j}}{\theta+n}&=\sum_{l=1}^L\mathbb{P}(X_{n+1}\in C_{j_l}|X_1,\ldots,X_n)\\
&=\sum_{l=1}^L\frac{\nu(C_{j_l})\sum_{i=1}^m\alpha_{C_i}+N_{n,C_{j_l}}}{\sum_{i=1}^m\alpha_{C_i}+n}=\frac{\nu(B_j)\sum_{i=1}^m\alpha_{C_i}+N_{n,B_j}}{\sum_{i=1}^m\alpha_{C_i}+n}\quad\mbox{a.s.}
\end{align*}
\endgroup
After simplification, we get $n(\theta-\sum_{i=1}^m\alpha_{C_i})\nu(B_j)=(\theta-\sum_{i=1}^m\alpha_{C_i})N_{n,B_j}$ a.s., for all $n=1,2,\ldots$, which implies that $\theta=\sum_{i=1}^m\alpha_{C_i}$. In any other case, we can consider the partition $\{B_i\cap C_l,i=1,\ldots,k,l=1,\ldots,m\}$ and conclude that the three partitions produce the same constant.

Let $A\in\mathcal{G}$. Arguing through the partition $\{A\cap B_j,A^c\cap B_j,j=1,\ldots,k\}$,
\begingroup\allowdisplaybreaks
\begin{align*}
\mathbb{P}(X_{n+1}\in A|X_1,\ldots,X_n)&=\sum_{j=1}^k\frac{\theta\nu(A\cap B_j)+N_{n,A\cap B_j}}{\theta+n}=\frac{\theta\nu(A)+\sumN\delta_{X_i}(A)}{\theta+n}\qquad\mbox{a.s.}
\end{align*}
\endgroup
Using that $\mathcal{G}$ is c.g. under $\nu$, we obtain that, as measures on $\mathbb{X}$,
\begin{equation}\label{proof:pure:eq2}
\mathbb{P}(X_{n+1}\in dx|X_1,\ldots,X_n)=\frac{\theta\nu(A)+\sumN\delta_{X_i}(dx)}{\theta+n}\qquad\mbox{a.s.}
\end{equation}

Let $A\in\mathcal{G}$ and $B\in\mathcal{X}$. It follows that
\begingroup\allowdisplaybreaks
\begin{align*}
\mathbb{E}\bigl[\mathbbm{1}_A(X_{n+1})\cdot\mathbb{P}(X_{n+1}\in B|\mathcal{G}_{n+1})\bigr]&=\nu(A\cap B)=\int_AR_x(B)\nu(dx)=\mathbb{E}[\mathbbm{1}_A(X_{n+1})\cdot R_{X_{n+1}}(B)],
\end{align*}
\endgroup
and, since $\omega\mapsto R_{X_{n+1}(\omega)}(B)$ is $\mathcal{G}_{n+1}$-measurable,
\[\mathbb{P}(X_{n+1}\in B|\mathcal{G}_{n+1})=R_{X_{n+1}}(B)\qquad\mbox{a.s.}\]
Applying \eqref{proof:pure:eq2}, we obtain, on a set of probability one,
\begingroup\allowdisplaybreaks
\begin{align*}
\mathbb{P}(X_{n+1}\in B|X_1,\ldots,X_n)&=\mathbb{E}\bigl[\mathbb{P}(X_{n+1}\in B|X_1,\ldots,X_n,\mathcal{G}_{n+1})|X_1,\ldots,X_n\bigr]\\
&=\mathbb{E}\bigl[\mathbb{P}(X_{n+1}\in B|\mathcal{G}_{n+1})|X_1,\ldots,X_n\bigr]\\
&=\mathbb{E}\bigl[R_{X_{n+1}}(B)|X_1,\ldots,X_n\bigr]\\
&=\int_\mathbb{X}R_x(B)\biggl\{\frac{\theta\nu(dx)+\sumN\delta_{X_i}(dx)}{\theta+n}\biggr\}\\
&=\frac{\theta\int_\mathbb{X}R_x(B)\nu(dx)+\sumN\int_\mathbb{X}R_x(B)\delta_{X_i}(dx)}{\theta+n}\\
&=\frac{\theta\nu(B)+\sumN R_{X_i}(B)}{\theta+n}.
\end{align*}
\endgroup
Note that the above result does not contradict \eqref{proof:pure:eq2} when $B\in\mathcal{G}$ because of \eqref{proof:pure:eq1}. Therefore, we can conclude that $(X_n)_{n\geq1}$ is an MVPS with parameters $(\theta,\nu,R)$.

Conversely, suppose that $(X_n)_{n\geq1}$ is an exchangeable MVPS with reinforcement kernel $R(\cdot)=\nu(\cdot\mid\mathcal{G})$. Recall from Theorem \ref{model:theorem} and the discussion around \eqref{model:equation:mvps:sufficientness} that any exchangeable MVPS admits such a representation. Moreover, we may assume that $R$ is a.e. proper, without loss of generality. Arguing as in \eqref{proof:pure:eq1}, we obtain, for all $A\in\mathcal{G}$, on a set of probability one,
\begingroup\allowdisplaybreaks
\begin{align*}
\mathbb{P}(X_{n+1}\in A|X_1,\ldots,X_n)&=\frac{\theta\nu(A)+\sumN R_{X_i}(A)}{\theta+n}=\frac{\theta\nu(A)+\sumN\delta_{X_i}(A)}{\theta+n}=\frac{\theta\nu(A)+N_{n,A}}{\theta+n};
\end{align*}
\endgroup
thus, $(X_n)_{n\geq1}$ satisfies \eqref{pure:result:main:eq1}. Let $A_1,\ldots,A_n,B\in\mathcal{X}$ and $A\in\mathcal{G}$. Then
\begingroup\allowdisplaybreaks
\begin{align*}
\mathbb{E}[\mathbbm{1}_{A_1}(X_1)\cdots\mathbbm{1}_{A_n}(X_n)\mathbbm{1}_{A}&(X_{n+1})\mathbbm{1}_{B}(X_{n+1})]\\
&\,=\mathbb{E}\bigl[\mathbbm{1}_{A_1}(X_1)\cdots\mathbbm{1}_{A_n}(X_n)\cdot\mathbb{P}(X_{n+1}\in A\cap B|X_1,\ldots,X_n)\bigr]\\
&\,=\mathbb{E}\biggl[\mathbbm{1}_{A_1}(X_1)\cdots\mathbbm{1}_{A_n}(X_n)\frac{\theta\nu(A\cap B)+\sumN R_{X_i}(A\cap B)}{\theta+n}\biggr]\\
&\overset{(a)}{=}\mathbb{E}\biggl[\mathbbm{1}_{A_1}(X_1)\cdots\mathbbm{1}_{A_n}(X_n)\frac{\theta\nu(A\cap B)+\sumN R_{X_i}(B)\delta_{X_i}(A)}{\theta+n}\biggr]\\
&\,=\mathbb{E}\biggl[\mathbbm{1}_{A_1}(X_1)\cdots\mathbbm{1}_{A_n}(X_n)\int_AR_x(B)\biggl\{\frac{\theta\nu(dx)+\sumN\delta_{X_i}(dx)}{\theta+n}\biggr\}\biggr]\\
&\,=\mathbb{E}\bigl[\mathbbm{1}_{A_1}(X_1)\cdots\mathbbm{1}_{A_n}(X_n)\cdot\mathbb{E}[\mathbbm{1}_A(X_{n+1})R_{X_{n+1}}(B)|X_1,\ldots,X_n]\bigr]\\
&\,=\mathbb{E}[\mathbbm{1}_{A_1}(X_1)\cdots\mathbbm{1}_{A_n}(X_n)\mathbbm{1}_A(X_{n+1})\cdot R_{X_{n+1}}(B)]\\
&\,=\mathbb{E}\bigl[\mathbb{P}(X_1\in A_1,\ldots,X_n\in A_n|\mathcal{G}_{n+1})\cdot\mathbbm{1}_A(X_{n+1})\cdot R_{X_{n+1}}(B)\bigr]\\
&\overset{(b)}{=}\mathbb{E}\bigl[\mathbb{P}(X_1\in A_1,\ldots,X_n\in A_n|\mathcal{G}_{n+1})\cdot\mathbbm{1}_A(X_{n+1})\cdot\mathbb{P}(X_{n+1}\in B|\mathcal{G}_{n+1})\bigr],
\end{align*}
\endgroup
where $(a)$ follows from $R_x(A\cap B)=\nu(A\cap B|\mathcal{G})(x)=\nu(B|\mathcal{G})(x)\delta_x(A)=R_x(B)\delta_x(A)$, for $\nu$-a.e. $x$, since $A\in\mathcal{G}$; and $(b)$ follows by taking $A_1=\cdots=A_n=\mathbb{X}$ and noticing that $R_{X_{n+1}}(B)$ is $\mathcal{G}_{n+1}$-measurable. Therefore, $X_{n+1}$ and $(X_1,\ldots,X_n)$ are conditionally independent given $\mathcal{G}_{n+1}$. This concludes the proof of Theorem \ref{pure:result:main}.
\end{proof}

\begin{proof}[Proof of Proposition \ref{result:k-color}]
Let us define, for every $n\geq0$, $l=1,\ldots,m$ and $n_l=0,1,\ldots,n$, the quantity
\[g_{n,l}(n_l):=\sum_{j\in D_l}f_{n,j}(n_l).\]
It follows from \eqref{equation:k-color:assumption} that $f_{n,j}(n_l)>0$, and so $g_{n,l}(n_l)>0$. Recall that $m(j)$ denotes, for each $j=1,\ldots,k$, the index $l\in\{1,\ldots,m\}$ such that $j\in D_l$. Define $Y_n:=m(X_n)$, for $n\geq1$. Then $(Y_n)_{n\geq1}$ is an exchangeable process with predictive distributions
\begin{equation}\label{proof:k-color:cluster:predictive}
\mathbb{P}(Y_{n+1}=l|Y_1,\ldots,Y_n)=P_n(D_l)=g_{n,l}(N_{n,l})\qquad\mbox{a.s., for }l=1,\ldots,m.
\end{equation}
Therefore, $(Y_n)_{n\geq1}$ satisfies Johnson's sufficientness postulate \eqref{intro:equation:johnson_sufficientness}, so from Theorem 2.1 and Corollary 2.1 in \cite{zabell1982}, there exist $\bar{w}_1,\ldots,\bar{w}_m>0$ such that
\begin{equation}\label{proof:k-color:cluster:result1}
g_{n,l}(n_l)=\frac{\bar{w}_l+n_l}{\bar{w}+n},
\end{equation}
where $\bar{w}:=\sum_{l=1}^m\bar{w}_l$; thus, in particular,
\begin{equation}\label{proof:k-color:cluster:result2}
\mathbb{P}_0(D_l)=\mathbb{P}(Y_1=l)=g_{0,l}(0)=\frac{\bar{w}_l}{\bar{w}}\qquad\mbox{for }l=1,\ldots,m.
\end{equation}

Let $j\in\{1,\ldots,k\}$ be such that $D_{m(j)}=\{j\}$. Then $f_{n,j}(n_{m(j)})=g_{n,m(j)}(n_{m(j)})$, so from \eqref{proof:k-color:cluster:result1}-\eqref{proof:k-color:cluster:result2},
\begin{equation}\label{proof:k-color:representation1}
\mathbb{P}_n(\{j\})=g_{n,m(j)}(N_{n,m(j)})=\frac{\bar{w}\cdot\mathbb{P}_0(\{j\})+\sumN\mathbbm{1}_{D_{m(j)}}(X_i)}{\bar{w}+n}\qquad\mbox{a.s.}
\end{equation}

Let $i,j\in\{1,\ldots,k\}$ be such that $m(i)=m(j)$. Fix $n\geq0$, where for $n=0$ the expressions involving $\mathbb{P}(\,\cdot\mid X_1,\ldots,X_n)$ are meant as unconditional statements. It follows for every $h=1,2,\ldots$ that
\begingroup\allowdisplaybreaks
\begin{align*}
\mathbb{P}(X_{n+1}=j,X_{n+2}\notin D_{m(j)}&,\ldots,X_{n+h+1}\notin D_{m(j)}|X_1,\ldots,X_n)\\
&=f_{n,j}(N_{n,m(j)})\prod_{s=1}^h\bigl(1-g_{n+s,m(j)}(N_{n,m(j)}+1)\bigr)\quad\mbox{a.s.},
\end{align*}
\endgroup
and
\begingroup\allowdisplaybreaks
\begin{align*}
\mathbb{P}(X_{n+1}\notin D_{m(j)},\ldots,X_{n+h}\notin &D_{m(j)},X_{n+h+1}=j|X_1,\ldots,X_n)\\
&=f_{n+h,j}(N_{n,m(j)})\prod_{s=0}^{h-1}\bigl(1-g_{n+s,m(j)}(N_{n,m(j)})\bigr)\quad\mbox{a.s.}
\end{align*}
\endgroup
By exchangeability, the two expressions are equal, so we obtain
\[\frac{f_{n,j}(N_{n,m(j)})}{f_{n+h,j}(N_{n,m(j)})}=\frac{\prod_{s=0}^{h-1}\bigl(1-g_{n+s,m(j)}(N_{n,m(j)})\bigr)}{\prod_{s=1}^h\bigl(1-g_{n+s,m(j)}(N_{n,m(j)}+1)\bigr)}\qquad\mbox{a.s.}\]
Since the right-hand side depends on $j$ only through $m(j)$, the same ratio $\frac{f_{n,i}}{f_{n+h,i}}$ for $i$ is equal to the same quantity; therefore, with $n_{m(j)}\equiv n_{m(i)}\in\{0,1,\ldots,n\}$ fixed,
\begin{equation}\label{proof:k-color:cluster:result3}
\frac{f_{n,i}(n_{m(i)})}{f_{n,j}(n_{m(j)})}=\frac{f_{n+h,i}(n_{m(i)})}{f_{n+h,j}(n_{m(j)})}\qquad\mbox{for every }h=1,2,\ldots
\end{equation}

On the other hand, for every $h=1,\ldots,n$,
\begingroup\allowdisplaybreaks
\begin{align*}
\mathbb{P}(X_1\notin D_{m(j)},\ldots,X_n&\notin D_{m(j)},X_{n+1}=j,X_{n+2}\in D_{m(j)},\ldots,X_{2n+1}\in D_{m(j)})\\
&=f_{n,j}(0)\prod_{s=0}^{n-1}\bigl(1-g_{s,m(j)}(0)\bigr)\prod_{s=1}^ng_{n+s,m(j)}(s)\qquad\mbox{a.s.},
\end{align*}
\endgroup
and, applying the permutation $\sigma$ of $\{1,\ldots,2n+1\}$ that swaps the indices $(1,\ldots,h)$ with $(n+2,\ldots,n+h+1)$,
\begingroup\allowdisplaybreaks
\begin{align*}
\mathbb{P}(X_{\sigma(1)}\notin D_{m(j)},&\ldots,X_{\sigma(n)}\notin D_{m(j)},X_{\sigma(n+1)}=j,\ldots,X_{\sigma(2n+1)}\in D_{m(j)})\\
&\begin{aligned}=f_{n,j}(h)&\prod_{s=0}^{h-1}g_{s,m(j)}(s)\prod_{s=0}^{n-h-1}\bigl(1-g_{h+s,m(j)}(h)\bigr)\\
&\times\prod_{s=1}^h\bigl(1-g_{n+s,m(j)}(h)\bigr)\prod_{s=1}^{n-h}g_{n+h+s,m(j)}(h+s-1)\quad\mbox{a.s.}\end{aligned}
\end{align*}
\endgroup
Since $(X_1,\ldots,X_{2n+1})\overset{d}{=}(X_{\sigma(1)},\ldots,X_{\sigma(2n+1)})$ by exchangeability, we obtain, as in \eqref{proof:k-color:cluster:result3}, with $n$ fixed,
\begin{equation}\label{proof:k-color:cluster:result4}
\frac{f_{n,i}(0)}{f_{n,j}(0)}=\frac{f_{n,i}(h)}{f_{n,j}(h)}\qquad\mbox{for every }h=1,\ldots,n.
\end{equation}
It follows from \eqref{proof:k-color:cluster:result3} and \eqref{proof:k-color:cluster:result4} that, for any such pair $i,j$, the ratio below is independent of $n$ and $n_{m(j)}\equiv n_{m(i)}$,
\[\frac{f_{n,i}(n_{m(i)})}{f_{n,j}(n_{m(j)})}=:d_{i,j}.\]
Summing over all $i$ in $D_{m(j)}$ and letting $w_j:=\frac{\bar{w}_{m(j)}}{\sum_{i\in D_{m(j)}}d_{i,j}}$, gives
\[f_{n,j}(n_{m(j)})=\frac{1}{\sum_{i\in D_{m(j)}}d_{i,j}}g_{n,m(j)}(n_{m(j)})=\frac{w_j}{\bar{w}_{m(j)}}g_{n,m(j)}(n_{m(j)});\]
thus, in particular, from \eqref{proof:k-color:cluster:result2},
\[\mathbb{P}_0(\{j\})=f_{0,j}(0)=\frac{w_j}{\bar{w}_{m(j)}}g_{0,m(j)}(0)=\frac{w_j}{\bar{w}}.\]
In that case, $\mathbb{P}_0(\{j\}|D_{m(j)})=\frac{w_j}{\bar{w}_{m(j)}}$; therefore, using \eqref{proof:k-color:cluster:result1},
\begingroup\allowdisplaybreaks
\begin{align}
\begin{aligned}
\mathbb{P}_n(\{j\})=f_{n,j}(N_{n,m(j)})&=\frac{w_j}{\bar{w}_{m(j)}}\frac{\bar{w}_{m(j)}+N_{n,m(j)}}{\bar{w}+n}\\
&=\frac{\bar{w}\cdot\mathbb{P}_0(\{j\})+\mathbb{P}_0(\{j\}|D_{m(j)})\sumN\mathbbm{1}_{D_{m(j)}}(X_i)}{\bar{w}+n}\qquad\mbox{a.s.}
\end{aligned}\label{proof:k-color:representation2}
\end{align}
\endgroup
We can now conclude from \eqref{proof:k-color:representation1} and \eqref{proof:k-color:representation2} that, as measures on $\mathbb{X}$,
\begingroup\allowdisplaybreaks
\begin{align*}
\mathbb{P}(X_{n+1}\in dy\mid X_1,\ldots,X_n)&=\frac{\bar{w}\cdot\mathbb{P}_0(dy)+\sumN\sum_{l=1}^m\mathbb{P}_0(dy\mid D_l)\cdot\mathbbm{1}_{D_l}(X_i)}{\bar{w}+n}\\
&=\frac{\bar{w}\cdot\mathbb{P}_0(dy)+\sumN\mathbb{P}_0(dy\mid\sigma(D_1,\ldots,D_m))(X_i)}{\bar{w}+n}\quad\mbox{a.s.},
\end{align*}
\endgroup
which implies that $(X_n)_{n\geq1}$ is an MVPS with parameters $\bigl(\bar{w},\mathbb{P}_0,\mathbb{P}_0(\cdot|\sigma(D_1,\ldots,D_m))\bigr)$.
\end{proof}

\begin{remark}
%Regarding the necessity of the condition $m\geq3$ in Proposition \ref{result:k-color}, observe that when $m=1$, then $D_1=\mathbb{X}$ and, for every $j=1,\ldots,k$ and $n\geq0$,
%\[\mathbb{P}_n(\{j\})=f_{n,j}(n),\]
%which is independent of $X_1,\ldots,X_n$. From there,
%\[\mathbb{P}(X_{n+2}=j|X_1,\ldots,X_n)=\mathbb{E}\bigl[\mathbb{P}_{n+1}(\{j\})\bigr| X_1,\ldots,X_n\bigr]=f_{n+1,j}(n+1)\qquad\mbox{a.s.},\]
%and, since $(X_1,\ldots,X_n,X_{n+1})\overset{d}{=}(X_1,\ldots,X_n,X_{n+2})$ by exchangeability,
%\[f_{n,j}(n)=\mathbb{P}(X_{n+2}=j|X_1,\ldots,X_n)=f_{n+1,j}(n+1)\qquad\mbox{a.s.}\]
%Moreover, $\mathbb{P}_0(\{j\})=\mathbb{P}(X_2=j)=\mathbb{E}\bigl[\mathbb{P}(X_2=j|X_1)\bigr]=f_{1,j}(1)$; therefore,
%\[\mathbb{P}(X_{n+1}=j|X_1,\ldots,X_n)=\mathbb{P}_0(\{j\}),\]
%which implies that $(X_n)_{n\geq1}$ is i.i.d. with marginal distribution $\mathbb{P}_0$.
Notice that when $m=2$, the sufficientness postulate \eqref{proof:k-color:cluster:predictive} at the level of the partition does not necessarily imply \eqref{proof:k-color:cluster:result1}, and so $(X_n)_{n\geq1}$ need not be an MVPS (see Remark 1 in \cite{zabell1982} and references therein). To obtain \eqref{proof:k-color:cluster:result1} when $m=2$, the authors in \cite{zabell1982} suggest assuming that the function $g_{n,l}$ in \eqref{proof:k-color:cluster:predictive} is linear in $n_l$.
\end{remark}

\begin{proof}[Proof of Proposition \ref{result:k-color:alternative}]
Let us define, for every $l=1,\ldots,m$ and $p_l\in\bigl\{\frac{\bar{w}_l+n_l}{\bar{w}+n}:n_l=0,\ldots,n;n=0,1,\ldots\bigr\}$, the quantity
\[g_l(p_l):=\sum_{j\in D_l}f_j(p_l).\]
Then a straightforward extension of the arguments in \cite{hill1987} to $m$ colors implies that $g_l(p_l)=p_l$. The rest of the proof follows the same steps as in the proof of Proposition \ref{result:k-color}, starting from \eqref{proof:k-color:cluster:result1}.
\end{proof} 

\subsection*{Acknowledgments}

%We are grateful to an Associate Editor and two anonymous referees for their valuable comments and helpful suggestions, which improved the quality of this work.

This study is financed by the European Union-NextGenerationEU, through the National Recovery and Resilience Plan of the Republic of Bulgaria, project No. BG-RRP-2.004-0008.

\bibliography{mybib} 

\end{document}